\documentclass{amsart}

\newcommand{\dosomething}[1]{\textbf{[#1]}}
\newcommand{\donothing}[1]{}
\newcommand{\Comment}{\donothing}
\newcommand{\Commentson}{\renewcommand{\Comment}{\dosomething}}

\Commentson

\usepackage[letterpaper,hmargin=1.25in,vmargin=1.22in]{geometry}

\usepackage{amssymb}

\usepackage[all]{xy}
\SilentMatrices %
\SelectTips{cm}{}
\newdir{ >}{{}*!/-5pt/@{>}}
\newdir{ (}{{}*!/-5pt/@^{(}}
\usepackage{tabularx}

\newcommand{\dfn}[1]{\textbf{\boldmath{#1}}}

\newtheorem{de}{Definition}[section]

\newtheorem{prop}[de]{Proposition}
\newtheorem{cor}[de]{Corollary}
\newtheorem{thm}[de]{Theorem}

\newtheorem*{claim*}{Claim}

\theoremstyle{remark}
\newtheorem{aside}[de]{Aside}
\newtheorem{rem}[de]{Remark}
\newtheorem{ex}[de]{Example}

\DeclareMathOperator*{\colim}{colim}
\newcommand{\ra}{\to}
\newcommand{\lra}{\longrightarrow}
\newcommand{\llra}[1]{\stackrel{#1}{\lra}}  %
\DeclareMathOperator{\obj}{Obj}
\DeclareMathOperator{\im}{Im}
\DeclareMathOperator{\Vect}{Vect}
\DeclareMathOperator{\dom}{Dom}

\newcommand{\Diff}{{\mathfrak{D}\mathrm{iff}}}
\newcommand{\Top}{{\mathfrak{T}\mathrm{op}}}

\newcommand{\DS}{{\mathcal{DS}}}
\newcommand{\cC}{{\mathcal{C}}}
\newcommand{\cD}{{\mathcal{D}}}
\newcommand{\cG}{{\mathcal{G}}}
\newcommand{\ddt}{{\textstyle\frac{d}{dt}}}
\newcommand{\VSD}{{\mathcal{VSD}}}
\newcommand{\DVS}{{\mathcal{DVS}}}
\newcommand{\pr}{{\mathrm{pr}}}
\newcommand{\id}{{\mathrm{id}}}

\def \N{\mathbb{N}}
\def \Z{\mathbb{Z}}
\def \Q{\mathbb{Q}}
\def \C{\mathbb{C}}
\def \R{\mathbb{R}}

\usepackage{ifpdf}
\ifpdf  \usepackage[pdftex,bookmarks=false]{hyperref}
\else   \usepackage[hypertex]{hyperref}
\fi

\title{Tangent spaces and tangent bundles \break for diffeological spaces}
\author{J. Daniel Christensen}
\email{jdc@uwo.ca}
\address{Department of Mathematics, University of Western Ontario, London, Ontario, Canada}
\author{Enxin Wu}
\email{enxin.wu@univie.ac.at}
\address{Faculty of Mathematics, University of Vienna, Vienna, Austria}
\date{November 14, 2015}

\begin{document}



\begin{abstract}
We study how the notion of tangent space can be extended from smooth
manifolds to diffeological spaces, which are generalizations of smooth
manifolds that include singular spaces and infinite-dimensional spaces.
We focus on two definitions.
The internal tangent space of a diffeological space is defined using
smooth curves into the space,
and the external tangent space is defined using smooth derivations on germs of smooth functions.
We prove fundamental results about these tangent spaces,
compute them in many examples, and observe
that while they agree for smooth manifolds and many of the examples, they
do not agree in general.

After this, we recall Hector's definition of the tangent bundle of a diffeological space,
and show that both scalar multiplication and addition can fail to be smooth,
revealing errors in several references.
We then give an improved definition of the tangent bundle, using what
we call the dvs diffeology, which ensures that scalar multiplication
and addition are smooth.
We establish basic facts about these tangent bundles, compute
them in many examples, and study the question of whether the fibres
of tangent bundles are fine diffeological vector spaces.

Our examples include singular spaces,
spaces whose natural topology is non-Hausdorff (e.g., irrational tori),
infinite-dimensional vector spaces and diffeological groups,
and spaces of smooth maps between smooth manifolds (including diffeomorphism groups).
\end{abstract}

\maketitle

\tableofcontents

\section{Introduction}

The notion of smooth manifold has been generalized in many ways, some
of which are summarized and compared in~\cite{St}.
Diffeological spaces were introduced by Souriau~\cite{So1,So2}
and are one such generalization, which includes as special cases
manifolds with corners, infinite-dimensional manifolds,
and a wide variety of spaces with complicated local behaviour.
In fact, the collection of diffeological spaces is closed under taking
subsets, quotients and function spaces, and thus gives rise to
a very well-behaved category.
Moreover, the definition of diffeological space is extremely simple,
and we encourage the reader not familiar with the definition to
read Definition~\ref{de:diffeological-space} now.
The standard textbook for diffeological spaces is~\cite{I3},
and we briefly summarize the basic theory in Section~\ref{s:basics}.

Tangent spaces and tangent bundles are important tools for studying
smooth manifolds.
There are many equivalent ways to define the tangent space of a smooth manifold
at a point.
These approaches have been generalized to diffeological spaces by many
authors, including the following.
In the papers that introduced diffeological groups and spaces,
Souriau~\cite{So1,So2} defined tangent spaces for diffeological groups
by identifying smooth curves using certain states.
Hector~\cite{He} defined tangent spaces and tangent bundles for all diffeological
spaces using smooth curves and a more intrinsic identification,
and these were developed further in~\cite{HM,La}.
(We point out some errors in all of~\cite{He,HM,La}, and give
corrected proofs when possible or counterexamples in other cases.)
Vincent~\cite{V} defined tangent spaces for diffeological spaces
by looking at smooth derivations represented by smooth curves,
and built associated tangent bundles using the same construction we use
in this paper.
Iglesias-Zemmour~\cite[6.53]{I3} defined the tangent space to a
diffeological space at a point as a subspace of the dual of the
space of $1$-forms at that point, and used these to define tangent bundles.

In this paper, we begin by studying two approaches to defining
the tangent space of a general diffeological space at a point.
The first is the approach introduced by Hector, which uses smooth curves,
and which we call the internal tangent space (see Subsection~\ref{ss:internal}).
The second is a new approach that uses smooth derivations on germs of smooth real-valued functions,
which we call the external tangent space (see Subsection~\ref{ss:external}).
In these subsections, we prove basic facts about these tangent spaces, such as locality,
and give tools that allow for the computations we do later.
We also show that the internal tangent space respects finite products,
and prove the non-trivial result that the internal tangent space
depends only on the plots of dimension at most $2$,
while the external tangent space depends only on the plots of dimension at most $1$.

In Subsections~\ref{ss:examples}, \ref{ss:dvs} and~\ref{ss:function}, we compute
these two tangent spaces for a diverse set of examples.
We summarize some of the computations in the following table:
\medskip
\begin{center}
\setlength\extrarowheight{1pt}
\renewcommand{\arraystretch}{1.1}
\begin{tabular}{|l|c|c|c|}
\hline
\bf Diffeological space & \bf Base point & \bf Internal & \bf External \\
\hline
discrete diffeological space & any point & $\R^0$ & $\R^0$ \\
\hline
indiscrete diffeological space & any point & $\R^0$ & $\R^0$ \\
\hline
topological space with continuous diffeology & any point & $\R^0$ & $\R^0$ \\
\hline
smooth manifold of dimension $n$ & any point & $\R^n$ & $\R^n$ \\
\hline
axes in $\R^2$ with the pushout diffeology & 0 & $\R^2$ & $\R^2$ \\
\hline
axes in $\R^2$ with the sub-diffeology & 0 & $\R^2$ & $\R^2$ \\
\hline
\parbox[c][2.5em]{2in}{\raggedright three lines intersecting at $0$ in $\R^2$\\[1pt]
with the sub-diffeology} & 0 & $\R^3$ & $\R^3$ \\
\hline
$\R^n$ with wire diffeology ($n \geq 2$) & any point
& \parbox[c][2.5em]{.75in}{\raggedright uncountable\\[-1pt] dimension} & $\R^n$ \\
\hline
$1$-dimensional irrational torus & any point & $\R^{\phantom{1}}$ & $\R^0$ \\
\hline
quotient space $\R^n/O(n)$ & $[0]$ & $\R^0$ & $\R^{\phantom{1}}$ \\
\hline
$[0,\infty)$ with the sub-diffeology of $\R$ & $0$ & $\R^0$ & $\R^{\phantom{1}}$ \\
\hline
vector space $V$ with fine diffeology & any point & $V^{\phantom{1}}$ & \\
\hline
\parbox[c][2.5em]{1.8in}{\raggedright diffeomorphism group of a\\[-1pt]
compact smooth manifold $M$}
& $1_M$ & \parbox[c][2.5em]{.75in}{\raggedright $C^\infty$ vector\\[-1pt] fields on $M$} & \\
\hline
\end{tabular}
\end{center}
\medskip
We see that these two tangent spaces coincide in many cases, including of
course for smooth manifolds, but that they are different in general.
In Subsection~\ref{ss:other-approaches}, we briefly describe some variants
of our definitions that one could also consider.

In Section~\ref{s:bundle}, we study tangent bundles.
Since the internal tangent space has better formal properties, and
we are able to compute it in more examples,
we define our tangent bundle (as a set) to be the disjoint union
of the internal tangent spaces.
Subsection~\ref{ss:dvs} begins by describing the diffeology that Hector
put on this internal tangent bundle~\cite{He},
and then shows that it is not well-behaved in general.
For example, we show in Example~\ref{ex:bundleofcross} that the fibrewise addition and scalar
multiplication maps are not smooth in general, revealing errors in~\cite{He,HM,La}.
We then introduce a refinement of Hector's diffeology, which we call the dvs diffeology,
that avoids these problems.
We also reprove the fact that the internal tangent bundle of a diffeological group
with Hector's diffeology is trivial, since the original proof was partially based on a
false result, and as a result we conclude that Hector's diffeology and the dvs diffeology coincide
in this case.
In Subsection~\ref{ss:conceptual}, we give a conceptual explanation of the
relationship between Hector's diffeology and the dvs diffeology:  they are
colimits, taken in different categories, of the same diagram.

The two diffeologies on the tangent bundle give rise to
diffeologies on each internal tangent space.
In Subsection~\ref{ss:fine}, we study the question of when internal tangent spaces,
equipped with either of these diffeologies,
are fine diffeological vector spaces.  Here the \dfn{fine diffeology} on a vector
space is the smallest diffeology making the addition and scalar multiplication maps smooth.
We show that for many infinite-dimensional spaces,
both Hector's diffeology and the dvs diffeology on the internal tangent spaces are not fine.
On the other hand, we also show that the internal tangent space of any
fine diffeological vector space $V$ at any point
is isomorphic to $V$ as a diffeological vector space.
As a by-product, we show that the inverse function theorem does not hold for general diffeological spaces.

Finally, in Subsection~\ref{ss:function}, we study the internal tangent bundles of function spaces,
and generalize a result in~\cite{He,HM} that says that the internal tangent space of
the diffeomorphism group of a compact smooth manifold at the identity is isomorphic
to the vector space of all smooth vector fields on the manifold.
Again we find that in these cases, Hector's diffeology coincides with the dvs diffeology.

The paper~\cite{CW2} is a sequel to the present paper.  It proves that a
diffeological bundle gives rise to an exact sequence of tangent spaces,
and gives conditions under which Hector's diffeology and the dvs diffeology
on the tangent bundle agree.
\medskip

All smooth manifolds in this paper are assumed to be finite-dimensional,
Hausdorff, second countable and without boundary,
all vector spaces are assumed to be over the field $\R$,
and all linear maps are assumed to be $\R$-linear.

\section{Background on diffeological spaces}\label{s:basics}

We provide a brief overview of diffeological spaces.
All the material in this section can be found
in the standard textbook~\cite{I3}.
For a concise introduction to diffeological spaces, we recommend~\cite{CSW},
particularly Section~2 and the introduction to Section~3.

\begin{de}[\cite{So2}]\label{de:diffeological-space}
A \dfn{diffeological space} is a set $X$
together with a specified set $\cD_X$ of functions $U \ra X$ (called \dfn{plots})
for each open set $U$ in $\R^n$ and for each $n \in \N$,
such that for all open subsets $U \subseteq \R^n$ and $V \subseteq \R^m$:
\begin{enumerate}
\item (Covering) Every constant function $U \ra X$ is a plot;
\item (Smooth Compatibility) If $U \ra X$ is a plot and $V \ra U$ is smooth,
then the composite $V \ra U \ra X$ is also a plot;
\item (Sheaf Condition) If $U=\cup_i U_i$ is an open cover
and $U \ra X$ is a function such that each restriction $U_i \ra X$ is a plot,
then $U \ra X$ is a plot.
\end{enumerate}
We usually use the underlying set $X$ to denote the diffeological space $(X,\cD_X)$.

A function $f:X \rightarrow Y$ between diffeological spaces is
\dfn{smooth} if for every plot $p:U \ra X$ of $X$,
the composite $f \circ p$ is a plot of $Y$.
\end{de}

Write $\Diff$ for the category of diffeological spaces and smooth maps.
Given two diffeological spaces $X$ and $Y$,
we write $C^\infty(X,Y)$ for the set of all smooth maps from $X$ to $Y$.
An isomorphism in $\Diff$ will be called a \dfn{diffeomorphism}.

Every smooth manifold $M$ is canonically a diffeological space with the same
underlying set and plots taken to be all smooth maps $U \ra M$ in the usual sense.
We call this the \dfn{standard diffeology} on $M$, and, unless we say
otherwise, we always equip a smooth manifold with this diffeology.
It is easy to see that smooth maps in the usual sense between
smooth manifolds coincide with smooth maps between them with the standard diffeology.

The set $\cD$ of diffeologies on a fixed set $X$ is ordered by inclusion,
and is a complete lattice.
The largest element in $\cD$ is called the \dfn{indiscrete diffeology} on $X$,
and consists of all functions $U \ra X$.
The smallest element in $\cD$ is called the \dfn{discrete diffeology} on $X$,
and consists of all locally constant functions $U \ra X$.

The smallest diffeology on $X$ containing a set of functions
$\mathcal{A}=\{U_i \ra X\}_{i \in I}$ is called the diffeology \dfn{generated} by $\mathcal{A}$.
It consists of all functions $f:U \ra X$ that locally either factor through
the given functions via smooth maps, or are constant.
The standard diffeology on a smooth manifold is generated by any
smooth atlas on the manifold,
and for every diffeological space $X$, $\cD_X$ is generated by
$\cup_{n \in \N}\, C^\infty(\R^n,X)$.

For a diffeological space $X$ with an equivalence relation~$\sim$,
the smallest diffeology on $X/{\sim}$ making the quotient map $X \twoheadrightarrow X/{\sim}$ smooth
is called the \dfn{quotient diffeology}.
It consists of all functions $U \ra X/{\sim}$ that locally factor through the quotient map.
Using this, we call a smooth map $f:X \to Y$ a \dfn{subduction} if it induces a diffeomorphism
$X/{\sim} \to Y$, where $x \sim x'$ if and only if $f(x)=f(x')$,
and $X/{\sim}$ has the quotient diffeology.

For a diffeological space $Y$ and a subset $A$ of $Y$,
the largest diffeology on $A$ making the inclusion map $A \hookrightarrow Y$ smooth
is called the \dfn{sub-diffeology}.
It consists of all functions $U \ra A$ such that $U \ra A \hookrightarrow Y$ is a plot of $Y$.
Using this, we call a smooth map $f:X \ra Y$ an \dfn{induction} if it induces
a diffeomorphism $X \to \im(f)$, where $\im(f)$ has the sub-diffeology of $Y$.

The category of diffeological spaces is very well-behaved:

\begin{thm}
The category $\Diff$ is complete, cocomplete and cartesian closed.
\end{thm}

The descriptions of limits, colimits and function spaces are quite
simple, and are concisely described in~\cite[Section~2]{CSW}.
We will make use of these concrete descriptions.
Unless we say otherwise, every function space is equipped with the functional diffeology.

\medskip

We can associate to every diffeological space the following interesting topology:

\begin{de}[\cite{I1}]
Given a diffeological space $X$, the final topology induced by its plots,
where each domain is equipped with the standard topology,
is called the \dfn{$D$-topology} on $X$.
\end{de}

In more detail, if $(X, \cD)$ is a diffeological space,
then a subset $A$ of $X$ is open in the $D$-topology of $X$
if and only if $p^{-1}(A)$ is open for each $p \in \cD$.
We call such subsets \dfn{$D$-open}.

A smooth map $X \ra X'$ is continuous when $X$ and $X'$ are equiped with the
$D$-topology, and so this defines a functor $D: \Diff \ra \Top$ to the
category of topological spaces.

\begin{ex}
(1) The $D$-topology on a smooth manifold coincides with the usual topology.

(2) The $D$-topology on a discrete diffeological space is discrete,
and the $D$-topology on an indiscrete diffeological space is indiscrete.
\end{ex}

For more discussion of the $D$-topology of a diffeological space, see~\cite{CSW}.
\medskip

We will make use of the concept of diffeological group at
several points, so we present it here.

\begin{de}\label{def:diff-group}
A \dfn{diffeological group} is a group object in $\Diff$.
That is, a diffeological group is both a diffeological space and a group
such that the group operations are smooth maps.
\end{de}

A smooth manifold of dimension $n$ is formed by gluing together
open subsets of $\R^n$ via diffeomorphisms.
A diffeological space is also formed by gluing together open subsets
of $\R^n$ via smooth maps,
possibly for all $n \in \N$.
To make this precise, let $\DS$ be the category with
objects all open subsets of $\R^n$ for all $n \in \N$
and morphisms smooth maps between them.
Given a diffeological space $X$, we define $\DS/X$ to be the category with
objects all plots of $X$ and morphisms the commutative triangles
\[
\xymatrix@C5pt{U \ar[dr]_p \ar[rr]^f & & V \ar[dl]^q \\ & X, }
\]
with $p,q$ plots of $X$ and $f$ a smooth map.
We call $\DS/X$ the \dfn{category of plots of $X$}. Then we have:

\begin{prop}[{\cite[Proposition~2.7]{CSW}}]\label{colim}
The colimit of the functor $F:\DS/X \ra \Diff$
sending the above commutative triangle to $f:U \ra V$ is $X$.
\end{prop}

Given a diffeological space $X$, the category $\DS/X$ can be used to define
geometric structures on $X$.
For example, see~\cite{I3} for a discussion of differential forms and the
de Rham cohomology of a diffeological space.
In this paper, we will use a pointed version of $\DS/X$ to
define the internal tangent space of $X$.

\section{Tangent spaces}

We discuss two approaches to defining the tangent space of a diffeological space at a point:
the internal tangent space introduced by Hector using plots,
and the external tangent space defined using smooth derivations of the
algebra of germs of smooth functions.
We prove basic facts about these tangent spaces in Subsections~\ref{ss:internal}
and~\ref{ss:external}, and then compute them for a wide variety of examples
in Subsection~\ref{ss:examples}.
Although they are isomorphic for smooth manifolds,
we find that the two approaches are different for a general diffeological space;
see Examples~\ref{ex:wirediffeology}, \ref{irrtorus},
\ref{halfline:quotient} and~\ref{halfline:sub}.
In Subsection~\ref{ss:other-approaches}, we mention some other
approaches to defining tangent spaces.

\subsection{Internal tangent spaces}\label{ss:internal}

The internal tangent space of a pointed diffeological space is defined using plots.
It was first introduced in~\cite{He},
and is closely related to the kinematic tangent space of~\cite{KM}
(see Subsection~\ref{ss:other-approaches}).

\medskip

To start, we will define a pointed analog of the category $\DS/X$ of plots
of $X$, introduced just before Proposition~\ref{colim}.
Let $\DS_0$ be the category with
objects all connected open neighbourhoods of $0$ in $\R^n$ for all $n \in \N$
and morphisms the smooth maps between them sending $0$ to $0$.
Given a pointed diffeological space $(X,x)$,
we define $\DS_0/(X,x)$ to be the category
with objects the plots $p:U \ra X$ such that $U$ is connected, $0 \in U$ and $p(0)=x$,
and morphisms the commutative triangles
\[
\xymatrix@C5pt{U \ar[dr]_p \ar[rr]^f & & V \ar[dl]^q \\ & X , }
\]
where $p,q \in \obj(\DS_0/(X,x))$ and $f$ is a smooth map with $f(0)=0$.
We call $\DS_0/(X,x)$ the \dfn{category of plots of $X$ centered at $x$}.
It is the comma category of the natural functor from $\DS_0$ to
$\Diff_{*}$, the category of pointed diffeological spaces.

\begin{de}[\cite{He}]
Let $(X,x)$ be a pointed diffeological space.
The \dfn{internal tangent space} $T_x(X)$ of $X$ at $x$ is
the colimit of the composite of functors
$\DS_0/(X,x) \ra \DS_0 \ra \Vect$,
where $\Vect$ denotes the category of vector spaces and linear maps,
the first functor is the forgetful functor
and the second functor is given by $(f:U \ra V) \mapsto (f_*:T_0(U) \ra T_0(V))$.
Given a plot $p : U \to X$ sending $0$ to $x$ and an element $u \in T_0(U)$,
we write $p_*(u)$ for the element these represent in the colimit.
\end{de}

Let $f:(X,x) \ra (Y,y)$ be a smooth map between pointed diffeological spaces.
Then $f$ induces a functor $\DS_0/(X,x) \ra \DS_0/(Y,y)$.
Therefore, we have a functor $T:\Diff_* \ra \Vect$.
Indeed, the functor $T$ is the left Kan extension of the functor
$\DS_0 \ra \Vect$ sending $f:U \ra V$ to $f_*:T_0(U) \ra T_0(V)$
along the inclusion functor $\DS_0 \ra \Diff_*$.

Note that the category of plots of a diffeological space centered at a point
is usually complicated.
In order to calculate the internal tangent space at a point efficiently,
we need a simpler indexing category.

Let $(X,x)$ be a pointed diffeological space.
We define a category $\cG(X,x)$ whose objects are the objects in $\DS_0/(X,x)$,
and whose morphisms are germs at 0 of morphisms in $\DS_0/(X,x)$.
In more detail, morphisms from $p:(U,0) \ra (X,x)$ to $q:(V,0) \ra (X,x)$ in $\cG(X,x)$
consist of equivalence classes of smooth maps $f : W \to V$,
where $W$ is an open neighborhood of $0$ in $U$ and $p|_W = q \circ f$.
Two such maps are equivalent if they agree on an open neighborhood of $0$ in $U$.
Then there is a functor $\cG(X,x) \ra \Vect$ sending the morphism
\[
\xymatrix@C5pt{U \ar[rr]^{[f]} \ar[dr] && V \ar[dl] \\
                                      & X}
\]
in $\cG(X,x)$ to $f_*:T_0(U) \ra T_0(V)$, and its colimit is $T_x(X)$.
A \dfn{local generating set of $X$ at $x$} is a subset $G$ of $\obj(\cG(X,x))$,
such that for each object $p:(U,0) \ra (X,x)$ in $\cG(X,x)$,
there exist an element $q:(W,0) \ra (X,x)$ in $G$
and a morphism $p \ra q$ in $\cG(X,x)$.
A \dfn{local generating category of $X$ at $x$} is a subcategory $\cG$ of $\cG(X,x)$
such that $\obj(\cG)$ is a local generating set of $X$ at $x$.

\begin{prop}\label{lgcat}
If $\cG$ is a local generating category of $X$ at $x$, then
there is a natural epimorphism from
$\colim (\cG \hookrightarrow \cG(X,x) \ra \Vect)$ to $T_x(X)$.
Moreover, if $\cG$ is a final subcategory of $\cG(X,x)$, then this is an isomorphism.
\end{prop}
\begin{proof}
The first statement follows from the definition of a local generating category of $X$ at $x$,
and the second statement follows from~\cite[Theorem~IX.3.1]{Mac}.
\end{proof}

A \dfn{local generating set of curves of $X$ at $x$} is
a subset $C$ of $\obj(\cG(X,x)) \cap C^\infty(\R,X)$
such that for each object $p:\R \ra X$ in $\cG(X,x)$,
there exist an element $q:\R \ra X$ in $C$
and a morphism $p \ra q$ in $\cG(X,x)$.
A \dfn{local generating category of curves of $X$ at $x$}
is a subcategory $\cC$ of $\cG(X,x)$
such that $\obj(\cC)$ is a local generating set of curves of $X$ at $x$.

\begin{prop}\label{lgcurve}
If $\cC$ is a local generating category of curves of $X$ at $x$, then
the natural map
$\colim (\cC \hookrightarrow \cG(X,x) \ra \Vect) \ra T_x(X)$
is an epimorphism.
\end{prop}

In particular, every internal tangent vector in $T_x(X)$ is
a linear combination of internal tangent vectors of the form
$p_*(\frac{d}{dt})$, where $p : \R \to X$ is a smooth curve with $p(0) = x$,
and $\frac{d}{dt}$ is the standard unit vector in $T_0(\R)$.

\begin{proof}
This is because for each $u \in T_0(U)$
there exists a pointed smooth map $f:(\R,0) \ra (U,0)$ such that $f_*(\frac{d}{dt})=u$.
\end{proof}

Moreover, the relations between internal tangent vectors are determined by the
two-dimensional plots:

\begin{prop}\label{prop:2plots}
Let $(X,x)$ be a pointed diffeological space.
Let $X'$ be the diffeological space with the same underlying set as $X$,
with diffeology generated by all plots $\R^2 \ra X$. %
Then the identity map $X' \ra X$ is smooth and induces an isomorphism
$T_x(X') \ra T_x(X)$.
\end{prop}

\begin{proof}
Since $\R$ is a retract of $\R^2$,
every plot $\R \to X$ factors through a plot $\R^2 \to X$,
and so $X'$ contains the same 1-dimensional plots as $X$.
Therefore, by Proposition~\ref{lgcurve}, the linear map $T_x(X') \ra T_x(X)$ is surjective.

To prove injectivity, we need a more concrete description of the internal tangent
spaces. From the description of $T_x(X)$ as a colimit indexed by the
category $\cG(X,x)$, we can describe $T_x(X)$ as a quotient vector space $F/R$.
Here $F = \oplus_p T_0(U_p)$, where the sum is indexed over plots $p : U_p \to X$
sending $0$ to $x$,
and $R$ is the span of the vectors of the form $(p,v) - (q, g_*(v))$,
where $p : U_p \to X$ and $q : U_q \to X$ are plots sending $0$ to $x$,
$g : (U_p,0) \to (U_q,0)$ is a germ of smooth maps with $p = q \circ g$ as germs at $0$,
$v$ is in $T_0(U_p)$,
$(p,v)$ denotes $v$ in the summand of $F$ indexed by $p$,
and $(q, g_*(v))$ denotes $g_*(v)$ in the summand of $F$ indexed by $q$.
Unless needed for clarity, we will write such a formal difference as simply $v - g_*(v)$,
and not repeat the conditions on $p$, $q$, $g$ and $v$.
We call $v - g_*(v)$ a \dfn{basic} relation, and call $U_p$ the \dfn{domain} of the relation.

Similarly, $T_x(X') = F'/R'$, where $F'$ and $R'$ are defined as above,
but restricting to plots that locally factor through plots of dimension $2$.
In this notation, the natural map $T_x(X') \to T_x(X)$ is induced
by the inclusion $F' \subseteq F$, and the surjectivity of this map
says that every element of $F$ is equal modulo $R$ to an element of $F'$.

As a start to proving injectivity, we first show that the basic relations
$v - g_*(v)$ are generated by those that have 1-dimensional domain $U_p$.
Choose a germ $f : (\R,0) \to (U_p,0)$ so that $f_*(\ddt) = v$.
Then
\[
\begin{aligned}
  & \left[ (p \circ f, -\ddt) - (p, f_*(-\ddt)) \right] + \left[ (q \circ g \circ f, \ddt) - (q, (g \circ f)_*(\ddt)) \right ] \\
 =& \left[ (p \circ f, -\ddt) + (p, v)          \right] + \left[ (p \circ f, \ddt)  - (q, g_*(v))        \right ] \\
 =&\ (p, v) - (q, g_*(v)) ,
\end{aligned}
\qquad\qquad
\vcenter{
\xymatrix@C-14pt@R-6pt{        &  \R \ar[ld]_f \ar[rd]^{g \circ f} & \\
           U_p \ar[dr]_p \ar[rr]^g     &&  U_q \ar[dl]^q  \\
                         &   X & }
}
\]
which shows that our given relation is a sum of relations with
domain $\R$. %
(This argument is similar to that of Proposition~\ref{lgcurve}, and shows
how that argument could be made more formal.)

We now know that a general element of $R$ is of the form
$r = \sum_i v_i - (g_i)_*(v_i)$, where $g_i : \R \to U_{q_i}$.
Next we show that $r$ can be written as a sum of basic
relations such that any plot $q : U_q \to X$ (sending $0$ to $x$) with
$\dim(U_q) > 2$ appears in at most one term of the sum.
Suppose $q$ is such a plot that appears in more than one term of $r$.
Without loss of generality, suppose $g_1$ and $g_2$ are the germs $(\R,0) \to (U_q,0)$
and $v_1$ and $v_2$ are the vectors in $T_0(\R)$ which give two such terms in $r$.
Let $i_1$ and $i_2$ be the inclusions of $\R$ into $\R^2$ as the $x$- and $y$-axes,
and define a germ $g : (\R^2,0) \to (U_q,0)$ by $g(x,y) = g_1(x) + g_2(y)$.
Then we have a commutative diagram
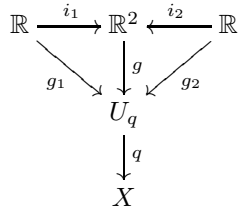
\[
  \xymatrix@R-6pt{ \R \ar[r]^{i_1} \ar[dr]_{g_1} & \R^2 \ar[d]^-g & \R \ar[l]_{i_2} \ar[dl]^{g_2} \\
                                                 &  U_q \ar[d]^q \\
                                                 &  X }
\]
from which it follows that
\[
  [ v_1 - (i_1)_*(v_1) ] + [ v_2 - (i_2)_*(v_2) ] + [ v - g_*(v) ] = [ v_1 - (g_1)_*(v_1) ] + [ v_2 - (g_2)_*(v_2) ] ,
\]
where $v := (i_1)_*(v_1) + (i_2)_*(v_2) \in T_0(\R^2)$.
The first two basic relations on the left-hand-side involve maps $\R \to \R^2$,
while the third involves a map $\R^2 \to U_q$.
Next, as in the previous paragraph, we replace this third basic relation
by two, one using a map $\R \to \R^2$ and the other a map $\R \to U_q$.
The result is that our new set of terms still consists of basic relations
with domain $\R$, but the number of occurrances of the plot $q$ has
been reduced by $1$.
Proceeding in this way, one can ensure that each such $q$ appears at
most once.

Finally, to prove injectivity of the map $F'/R' \to F/R$,
we need to show that every element $r$ of $F' \cap R$ is in $R'$.
By the above, we can write $r$ as a sum $r = \sum_i v_i - (g_i)_*(v_i)$,
where $g_i : \R \to U_{q_i}$ and each $q_i$ with domain of dimension
bigger than $2$ appears in at most one term.
Since $r$ is in $F'$, its component in any summand $T_0(U_q)$ of $F$
must be zero if $q$ does not locally factor through a plot of dimension $2$.
Since such a plot can appear in at most one term $v - g_*(v)$, no cancellation can
occur, so we must have that $g_*(v) = 0$.
So it suffices to show that we can eliminate such terms.
If $v = 0$, then $v - g_*(v) = 0$, so this term can be dropped.
Otherwise, we must have $g(0) = g'(0) = 0$, and so we can write
$g(x) = x^2 h(x)$ for a smooth map $h : \R \to U_q$.
Thus we can factor $g$ as
\[
  \xymatrix@C-10pt@R-6pt{ \R \ar[rr]^g \ar[dr]_-i && U_q \\ & \R^2 \ar[ur]_-f ,\!\!}
\]
where $i(x) = (x, x^2)$ and $f(s,t) = t \cdot h(s)$.
Then $v - g_*(v) = [v - i_*(v)] + [w - f_*(w)]$, where $w = i_*(v) = c \frac{d}{dx}$.
The map $i$ has codomain $\R^2$, so the first basic relation is of the
required form, but we must still deal with the second relation.
Consider the map $i_1 : \R \to \R^2$ sending $x$ to $(x,0)$.
Then $(i_1)_*(c \ddt) = w$, so
$w - f_*(w) = [-c\ddt - (i_1)_*(-c\ddt)] + [c\ddt - (f \circ i_1)_*(c \ddt)]$.
Again, the first basic relation is of the required form.
For the second, note that the map $f \circ i_1$ is the zero map, and so
$c \ddt - (f \circ i_1)_*(c \ddt) = c \ddt = c \ddt - k_*(c \ddt)$,
where $k : \R \ra \R^0$ is the zero map in the category $\cG(X,x)$.

This completes the proof.
\end{proof}

\begin{rem}
While the internal tangent space depends only on the two-dimensional plots, the
diffeology on the internal tangent bundle $TX$ that we define in Section~\ref{s:bundle}
contains all of the information about the diffeology on $X$, since $X$ is a retract of $TX$.
\end{rem}

The internal tangent space is local, in the following sense:

\begin{prop}\label{prop:internallocal}
Let $(X,x)$ be a pointed diffeological space,
and let $A$ be a $D$-open neighborhood of $x$ in $X$.
Equip $A$ with the sub-diffeology of $X$.
Then the natural inclusion map induces an isomorphism $T_x(A) \cong T_x(X)$.
\end{prop}
\begin{proof}
This is clear.
\end{proof}

We now investigate the tangent space of a product of diffeological spaces,
beginning with binary products.

\begin{prop}\label{product}
Let $(X_1,x_1)$ and $(X_2,x_2)$ be two pointed diffeological spaces.
Then there is a natural isomorphism of vector spaces
\[
T_{(x_1,x_2)}(X_1 \times X_2) \cong T_{x_1}(X_1) \times T_{x_2}(X_2).
\]
\end{prop}
\begin{proof}
The projections $\pr_j : X_1 \times X_2 \to X_j$, $j = 1, 2$, induce a natural map
\[ \alpha = ((\pr_1)_*, (\pr_2)_*) : T_{(x_1,x_2)}(X_1 \times X_2) \ra T_{x_1}(X_1) \times T_{x_2}(X_2). \]
Define inclusion maps $i_j : X_j \to X_1 \times X_2$ for $j = 1, 2$ by $i_1(y_1) = (y_1, x_2)$ and $i_2(y_2)=(x_1,y_2)$,
and consider the map
\[ \beta = (i_1)_* + (i_2)_* : T_{x_1}(X_1) \times T_{x_2}(X_2) \to T_{(x_1,x_2)}(X_1 \times X_2) \]
sending $(v_1, v_2)$ to $(i_1)_*(v_1) + (i_2)_*(v_2)$.
We claim that these maps are inverse to each other.

To check one of the composites, we first compute
\[
  (\pr_1)_*((i_1)_*(v_1) + (i_2)_*(v_2)) = (\pr_1)_* (i_1)_*(v_1) + (\pr_1)_* (i_2)_*(v_2)
= (\id_{X_1})_* (v_1) + (c_{x_1})_* (v_2) = v_1,
\]
where $c_{x_1} : X_2 \to X_1$ is the constant map at $x_1$ and hence $(c_{x_1})_* (v_2) = 0$.
Similarly, $(\pr_2)_*((i_1)_*(v_1) + (i_2)_*(v_2)) = v_2$, and so
$\alpha \circ \beta$ is the identity.

By Proposition~\ref{lgcurve}, it is enough to check the other composite
on internal tangent vectors of the form $p_*(\ddt)$, where $p = (p_1, p_2) : \R \to X_1 \times X_2$
and $p(0) = (x_1, x_2)$.
We have
\[
  \beta(\alpha(p_*(\ddt))) = (i_1)_* (\pr_1)_* p_*(\ddt) + (i_2)_* (\pr_2)_* p_* (\ddt)
= (i_1)_* (p_1)_* (\ddt) + (i_2)_* (p_2)_* (\ddt) = p_*(\ddt),
\]
where the last equality follows from the diagram
\[
  \xymatrix{
                                  & \R \ar[d]^-{\Delta = (1,1)}   & \\
    \R \ar[r]^-{i_x} \ar[d]_{p_1} & \R^2 \ar[d]^{p_1 \times p_2} & \R \ar[l]_-{i_y} \ar[d]^{p_2} \\
    X_1 \ar[r]_-{i_1}             & X_1 \times X_2               & X_2 , \ar[l]^-{i_2} \\
  }
\]
using that $(p_1 \times p_2) \circ \Delta = p$ and
$(i_x)_*(\ddt) + (i_y)_*(\ddt) = \frac{d}{dx} + \frac{d}{dy} = \Delta_*(\ddt)$.
\end{proof}

\begin{rem}\label{rem:infinite-product}
For an arbitrary product $X = \prod_{j \in J} X_j$ of diffeological spaces
and a point $x = (x_j)$,
there is also a natural linear map $\alpha: T_x(X) \ra \prod_{j \in J} T_{x_j}(X_j)$
induced by the projections.
We will characterize when $\alpha$ is surjective, and show that this
is not always the case.

To do so, we introduce some terminology.
An internal tangent vector of the form $p_*(\ddt)$ is said to be \dfn{representable},
and an internal tangent vector that can be expressed as a sum of $m$ or fewer
representables is said to be \dfn{$m$-representable}.
Recall that by Proposition~\ref{lgcurve},
every internal tangent vector in the domain of $\alpha$ is $m$-representable for some $m$.
Thus a family of internal tangent vectors $(v_j)$ in the image of $\alpha$ must have the property that each
component $v_j$ is $m$-representable for $m$ independent of $j$.
In fact, one can show that the image consists of exactly such families,
and therefore that the map $\alpha$ is surjective if and only if
there is an $m$ in $\N$ such that for all but finitely many $j$ in $J$,
every internal tangent vector in $T_{x_j}(X_j)$ is $m$-representable.
Manifolds and many other diffeological spaces have the property that
every internal tangent vector is $1$-representable (see the following remark),
so $\alpha$ is often surjective.

Here is an example for which $\alpha$ is not surjective.
For each $j \in \N$, consider the diffeological space $X_j$ which is
the quotient of $j$ copies of $\R$ where the origins have been identified
to one point $[0]$.  One can show that $T_{[0]}(X_j) \cong \R^{j}$ and contains internal tangent
vectors which are $j$-representable but not $(j-1)$-representable; see Example~\ref{ex:cross-pushout}.
So, with this family of diffeological spaces, $\alpha$ is not surjective.

We suspect that the map $\alpha$ can fail to be injective as well.  The injectivity
is related to the existence of a global bound on the number of basic relations
(see the proof of Proposition~\ref{prop:2plots}) needed to show that
two representable internal tangent vectors are equal.

See Example~\ref{ex:notfine}(\ref{item:inftyprod}) and Proposition~\ref{prop:isots} (with $X$ discrete)
for non-trivial cases in which $\alpha$ is an isomorphism.
\end{rem}

\begin{rem}\label{rem:representable}
Let $(X,x)$ be a pointed diffeological space.
If every internal tangent vector in $T_x(X)$ is $m$-representable,
then we say that $T_x(X)$ is \dfn{$m$-representable}.
This is the case, for example, when the vector space $T_x(X)$ is $m$-dimensional.
We will show that in fact many diffeological spaces have $1$-representable
internal tangent spaces.

Consider the case where $X$ is a smooth manifold.
Then the internal tangent space agrees with the usual tangent space
defined using curves, so $T_x(X)$ is $1$-representable for any $x$ in $X$.

Remark~\ref{rem:diff-group} shows that the internal tangent space of a
diffeological group at any point is also $1$-representable.
In particular, this holds for a diffeological vector space.
It is easy to see that if $A \subseteq X$ is either $D$-open in $X$ or a
retract of $X$, $x$ is in $A$ and $T_x(X)$ is $m$-representable, then so is $T_x(A)$.
In particular, the proof of Proposition~\ref{prop:cpt} shows that if
$Y$ is a diffeological space with compact $D$-topology
and $N$ is a smooth manifold, then the internal tangent space of
$C^\infty(Y,N)$ at any point is $1$-representable.

Finally, one can also show that the internal tangent space of a homogeneous
diffeological space (see~\cite[Definition~4.33]{CW}) at any point is $1$-representable.
\end{rem}

\subsection{External tangent spaces}\label{ss:external}

In contrast to the internal tangent space, which is defined using plots,
the external tangent space of a pointed diffeological space $(X,x)$
is defined using germs of real-valued smooth functions on $X$.
This is analogous to the operational tangent space defined in~\cite{KM}.

\medskip

Let $G_x(X) = \colim_B C^{\infty}(B, \R)$ be the diffeological space of \dfn{germs of
smooth functions of $X$ at $x$}, where the colimit is taken in $\Diff$,
$B$ runs over all $D$-open subsets of $X$ containing $x$ together with the sub-diffeology,
$C^\infty(B,\R)$ has the functional diffeology,
and the maps in the colimit are restrictions along inclusions.
$G_x(X)$ is a diffeological $\R$-algebra under pointwise addition,
pointwise multiplication and pointwise scalar multiplication,
i.e., all these operations are smooth,
and the evaluation map $G_x(X) \ra \R$ sending $[f]$ to $f(x)$
is a well-defined smooth $\R$-algebra map.

\begin{de}
An \dfn{external tangent vector} on $X$ at $x$
is a smooth
derivation on $G_x(X)$.  That is, it is a smooth
linear map $F:G_x(X) \ra \R$
such that the \dfn{Leibniz rule} holds: $F([f][g])=F([f])g(x)+f(x)F([g])$.
The \dfn{external tangent space} $\hat{T}_xX$ is the set of
all external tangent vectors of $X$ at $x$.
\end{de}

Clearly $\hat{T}_xX$ is a vector space under pointwise addition
and pointwise scalar multiplication.
The Leibniz rule implies that $F([c])=0$ for every external tangent vector
$F$ on $X$ at $x$ and every constant function $c: X \ra \R$.

Let $f:(X,x) \ra (Y,y)$ be a pointed smooth map between two pointed diffeological spaces.
Then $f$ induces a canonical linear map
$f_*:\hat{T}_x(X) \ra \hat{T}_y(Y)$ with $f_*(F)([g])=F([g \circ f])$
for $F \in \hat{T}_x(X)$ and $[g] \in G_y(Y)$,
where $g \circ f$ really means $g \circ f|_{f^{-1}(\dom(g))}$.
This gives a functor $\hat{T}:\Diff_* \ra \Vect$.

\medskip

We next give an equivalent characterization of the external tangent space of a diffeological space.
Recall that $G_x(X)$ is a diffeological $\R$-algebra,
and the evaluation map $G_x(X) \ra \R$ is a smooth $\R$-algebra homomorphism.
Let $I_x(X)$ be the kernel of the evaluation map, equipped with the sub-diffeology,
and let $I_x(X)/I^2_x(X)$ be the quotient vector space with the quotient diffeology.
We define $T'_x(X)=L^\infty(I_x(X)/I^2_x(X),\, \R)$, the set of all smooth linear maps $I_x(X)/I^2_x(X) \ra \R$.
It is a vector space.

\begin{prop}\label{equivexternal}
The map $\alpha:\hat{T}_x(X) \ra T'_x(X)$ defined by $\alpha(F)([f]+I_x^2(X))=F([f])$ is an isomorphism.
\end{prop}

\begin{proof}
One can define $\beta:T'_x(X) \ra \hat{T}_x(X)$ by $\beta(G)([g])=G([g]-[g(0)]+I_x^2(X))$,
where $G \in T'_x(X)$ and $[g] \in G_x(X)$.
It is easy to check that both $\alpha(F):I_x(X)/I_x^2(X) \ra \R$
and $\beta(G):G_x(X) \ra \R$ are smooth,
and that $\alpha$ and $\beta$ are well-defined inverses to each other.
\end{proof}

The external tangent space is also local:

\begin{prop}\label{extlocal}
Let $(X,x)$ be a pointed diffeological space, and
let $A$ be a $D$-open subset of $X$ containing $x$.
Equip $A$ with the sub-diffeology of $X$.
Then the natural inclusion map induces an isomorphism $\hat{T}_x(A) \cong \hat{T}_x(X)$.
\end{prop}
\begin{proof}
This is clear.
\end{proof}

Moreover, the external tangent space is determined by the one-dimensional plots:

\begin{prop}\label{prop:1-plots-determine-external}
Let $(X,x)$ be a pointed diffeological space,
and write $X'$ for the set $X$ with the diffeology generated by $C^\infty(\R,X)$.
Then the natural smooth map $X' \ra X$ induces an isomorphism $\hat{T}_x(X') \ra \hat{T}_x(X)$.
\end{prop}

\begin{proof}
By~\cite[Theorem~3.7]{CSW}, we know that the $D$-topology on $X'$ coincides with the $D$-topology
on $X$, as they have the same smooth curves. Let $B$ be a $D$-open subset of $X$, equipped 
with the sub-diffeology of $X$, and write $B'$ for the same set equipped with the sub-diffeology 
of $X'$.
The natural smooth map $1:B' \ra B$ induces a smooth map 
$1^*: C^\infty(B,\R) \ra C^\infty(B',\R)$ of $\R$-algebras.
We will show that this is a diffeomorphism, which then implies that $G_x(X) \cong G_x(X')$
and therefore that $\hat{T}_x(X') \cong \hat{T}_x(X)$.
The map $1^*$ is clearly injective.
Surjectivity follows from Boman's theorem~\cite[Corollary~3.14]{KM},
which says that a map $U \to \R$ is smooth if and only if it sends smooth curves to smooth curves,
where $U$ is open in some $\R^n$.
To see that the inverse of $1^*$ is smooth, note that the plots
$U \to C^{\infty}(B, \R)$ correspond to smooth maps $U \times B \to \R$.
Applying Boman's theorem again, we see that these are the same as the
plots of $C^{\infty}(B', \R)$.
\end{proof}

\subsection{Examples and comparisons}\label{ss:examples}

We now calculate the internal and external tangent spaces of
some pointed diffeological spaces.
A table summarizing the results is in the Introduction.

\begin{ex}\label{(in)dis}
(1) Let $X$ be a discrete diffeological space.
Then for each $x \in X$,
$T_x(X)=0$, since the local generating category of $X$ at $x$
with one object $x:\R^0 \ra X$ is final in $\cG(X,x)$.
Also, $\hat{T}_x(X)=0$, since $I_x(X)=0$.

(2) Let $X$ be an indiscrete diffeological space.
Then for each $x \in X$, $T_x(X)=0$, since for any $p:\R \ra X \in \obj(\DS_0/(X,x))$,
there exists $q:\R \ra X \in \obj(\DS_0/(X,x))$ such that $q \circ f=p$,
where $f:\R \ra \R$ is defined by $f(x)=x^3$.
Also, one can show easily that $\hat{T}_x(X)=0$.
\end{ex}

\begin{ex}\label{top}
Let $(X,x)$ be a pointed topological space.
Write $C(X)$ for the diffeological space with underlying set $X$
whose plots are the continuous maps.
Then~\cite[Proposition~4.3]{He} says that $T_x(C(X))=0$.

We now show that $\hat{T}_x(C(X))=0$ as well.
Let $A$ be a $D$-open subset of $C(X)$, equipped with the sub-diffeology,
and fix a smooth map $g : A \to \R$.
It suffices to show that $g$ is locally constant.
Let $p : \R \to A$ be a plot of $A$;
that is, $p$ is a continuous map $\R \to X$ whose image is in $A$.
Since $g$ is smooth, so is $g \circ p : \R \to \R$.
Moreover, for any continuous $h : \R \to \R$, $p \circ h$ is a plot of $A$
and so $g \circ p \circ h$ is smooth as well.
Taking $h(t) = |t|+a$ for each $a \in \R$, for example,
one sees that $g \circ p$ is constant and so $g$ is locally constant.
\end{ex}

\begin{ex}\label{smoothmanifolds}
It is a classical result that for every pointed smooth manifold $(X,x)$,
$T_x(X) \cong \R^n \cong \hat{T}_x(X)$, with $n = \dim(X)$.
In fact, any derivation $F:G_x(X) \ra \R$ is smooth.
This follows from Proposition~\ref{extlocal} and~\cite[Lemma~4.3]{CSW}.
\end{ex}

\begin{ex}\label{ex:cross-pushout}
Let $X$ be the pushout of $\xymatrix{\R & \R^0 \ar[l]_0 \ar[r]^0 & \R}$ in $\Diff$.
The commutative diagram
\[
\xymatrix{\R^0 \ar[r]^0 \ar[d]_0 & \R \ar[d]^{i_2} \\ \R \ar[r]_{i_1} & \R^2}
\]
with $i_1(s)=(s,0)$ and $i_2(t)=(0,t)$ induces a smooth injective map $i:X \ra \R^2$,
and we identify points in $X$ with points in $\R^2$ under the map $i$.
Note that the diffeology on $X$ is different from the sub-diffeology of $\R^2$
(see Example~\ref{ex:cross-sub}),
but the $D$-topology on $X$ is the same as the sub-topology of $\R^2$.
It is not hard to check that
\[
T_x(X) = \begin{cases}\R,   & \text{if $x \neq (0,0)$} \\
                      \R^2, & \text{if $x=(0,0)$.}
         \end{cases}
\]
We claim that the same is true for the external tangent spaces $\hat{T}_x(X)$:
\[
\hat{T}_x(X) = \begin{cases} \R,   & \text{if $x \neq (0,0)$} \\
                             \R^2, & \text{if $x=(0,0)$.}
               \end{cases}
\]
The first equality follows from
Proposition~\ref{extlocal} and Example~\ref{smoothmanifolds}.
For the second equality, we define maps
$a:\hat{T}_{(0,0)}(X) \ra \hat{T}_0(\R) \oplus \hat{T}_0(\R)$ by
$F \mapsto (F_1,F_2)$ with $F_k([f])=F([\tilde{f}_k])$,
where $\tilde{f}_k(x_1,x_2)=f(x_k)$ for $k=1,2$,
and $b:\hat{T}_0(\R) \oplus \hat{T}_0(\R) \ra \hat{T}_{(0,0)}(X)$ by
$(G_1,G_2) \mapsto G$ with $G([g])=G_1([g \circ j_1])+G_2([g \circ j_2])$,
where $j_1,j_2$ are the structural maps from the pushout diagram of $X$.
It is clear that both $a$ and $b$ are well-defined linear maps and
that they are inverses, so the second equality follows.

By the same method, one can show that if $X_j$ is the quotient of $j$ copies of $\R$
with the origins identified to one point $[0]$, then
\[
T_x(X_j)=\hat{T}_x(X_j)=\begin{cases} \R, & \textrm{if $x \neq [0]$} \\ \R^j, & \textrm{if $x=[0]$.} \end{cases}
\]
\end{ex}

\begin{rem}\label{re:cross-pushout}
Note that in the above example $\dim_{(0,0)}(X)=1<2=\dim(T_{(0,0)}(X))$;
see~\cite{I2} for the definition of the dimension of a diffeological space at a point.
In general, unlike smooth manifolds, there is no relationship between the dimension of a diffeological space
at a point and the dimension of its tangent space at that point.
\end{rem}

\begin{ex}\label{ex:cross-sub}
Let $Y=\{(x,y) \in \R^2 \mid xy=0\}$ with the sub-diffeology of $\R^2$.
The map $i$ introduced in Example~\ref{ex:cross-pushout} gives a smooth bijection $X \ra Y$.
However, this map is not a diffeomorphism.
To see this, let $f:\R \ra \R$ be defined by
\[
f(x)=\begin{cases} e^{-\frac{1}{x}}, & \textrm{if $x>0$} \\
                   0,                & \textrm{if $x \leq 0$.}
     \end{cases}
\]
Then $f$ is a smooth function, and $\R \ra \R^2$ defined by $x \mapsto (f(x),f(-x))$ induces a plot of $Y$, but not a plot of $X$.

Although $X$ and $Y$ are not diffeomorphic, their internal and
external tangent spaces at any point are isomorphic.
This is clear away from the origin, so consider
$y=(0,0) \in Y$.
Both \cite[Example~4.4(i)]{He} and \cite[Example~6.2]{HM} claim without proof that
$T_y(Y)=\R^2$.
We sketch a proof here.
The inclusion map $i:Y \ra \R^2$ induces a linear map $i_*:T_y(Y) \ra T_y(\R^2)=\R^2$.
There is also a linear map $f:\R^2 \ra T_y(Y)$ sending $(1,0)$ to $p_*(\frac{d}{dt})$
and sending $(0,1)$ to $q_*(\frac{d}{dt})$,
where $p:\R \ra Y$ is defined by $x \mapsto (x,0)$ and
$q:\R \ra Y$ is defined by $x \mapsto (0,x)$.
It is clear that $i_* \circ f=1_{\R^2}$, so $f$ is injective.
To prove that $f$ is surjective, it is enough to show that
if $r:\R \ra Y$ is a plot sending $0$ to $y$ and
$r$ does not locally factor through $p$ or $q$ near $0$, then $r_*(\frac{d}{dt})=0$.
It is easy to observe that if we write the composite $i \circ r:\R \ra \R^2$
as $(f_1,f_2)$, then $f_1(0)=f_2(0)=0=f_1'(0)=f_2'(0)$,
and hence the conclusion follows from an argument similar to that used in Example~\ref{halfline:sub} below.

Now we claim that $G_y(Y)=G_y(X)$ as diffeological spaces, from which it follows that
$\hat{T}_y(Y) = \hat{T}_y(X)$, which is $\R^2$ by Example~\ref{ex:cross-pushout}.
To see that $G_y(Y)=G_y(X)$, we first note that the $D$-topologies on $X$ and $Y$ are equal,
both being the sub-topology of $\R^2$.
There is a cofinal system of $D$-open neighbourhoods of $y$ each of which is
diffeomorphic to $Y$ (resp.\ $X$) when given the sub-diffeology of $Y$ (resp.\ $X$).
Thus it is enough to show that $C^\infty(Y,\R)=C^\infty(X,\R)$ as diffeological spaces.
There is a canonical smooth injection $i: C^\infty(Y,\R) \ra C^\infty(X,\R)$ induced by the smooth bijection $X \ra Y$.
To show that this is a diffeomorphism, it suffices to show that every plot $V \ra C^\infty(X,\R)$
factors through $i$.
By adjointness, this is equivalent to showing that every smooth map $V \times X \to \R$
factors through the bijection $V \times X \to V \times Y$.
A smooth map $V \times X \to \R$ is the same
as a pair of smooth maps $g, h : V \times \R \to \R$ such that $g(v,0) = h(v,0)$ for all $v \in V$.
Such a map extends to a smooth map $V \times \R^2 \to \R$ which sends
$(v, (x,y))$ to $g(v,x) + h(v,y) - g(v,0)$, and therefore the restriction to $V \times Y$
is smooth.
(See~\cite[Example 12(c)]{V} for a similar argument.)

By the same method, one can show that if $Y_j$ is the union of all coordinate axes in $\R^j$ with the sub-diffeology, then
\[
T_y(Y_j)=\hat{T}_y(Y_j)=\begin{cases} \R, & \textrm{if $y \neq 0$}\\ \R^j, & \textrm{if $y=0$}. \end{cases}
\]
\end{ex}

\begin{ex}\label{ex:axes-in-R3}
Let $A = Y_3$ from the previous example, and
let $B = \{ (x,y) \in \R^2 \mid x = 0 \text{ or}\break y = 0 \text{ or } x = y \}$,
a union of three lines through the origin, with the sub-diffeology of $\R^2$.
We prove below that $A$ and $B$ are diffeomorphic.
This is also proved in~\cite[Example~2.72]{Wa}, using a more complicated argument.
It follows that $T_0(B) = \hat{T}_0(B) = \R^3$,
which shows that the induction $B \hookrightarrow \R^2$ does not induce a
monomorphism under $T_0$ or $\hat{T}_0$.

Consider the smooth function $\R^3 \to \R^2$ sending $(x,y,z)$ to
$(x+z/\sqrt{2},\, y+z/\sqrt{2})$.
This restricts to a smooth bijection $A \to B$ sending
$(x, 0, 0)$ to $(x, 0)$, $(0, y, 0)$ to $(0, y)$, and
$(0, 0, z)$ to $(z, z)/\sqrt{2}$.
Write $f : B \to A$ for its inverse.  We will show that $f$ is smooth.
By Boman's theorem~\cite[Corollary~3.14]{KM},
it is enough to check that for any plot $p : \R \to B$, $f \circ p$ is smooth.
And for this, it is enough to check that if we regard $f \circ p$ as a map
$\R \to \R^3$, all derivatives exist at all points.
This follows from the following claim, where we also regard
the derivatives of $p$ as maps $\R \to \R^2$.
\end{ex}

\begin{claim*}
For each $t$ in $\R$, $p^{(k)}(t)$ is in $B$ and $(f \circ p)^{(k)}(t) = f(p^{(k)}(t))$.
\end{claim*}

\begin{proof}
We first prove this for $k = 1$.
If $p(t) \neq (0,0)$ or $p'(t) \neq (0,0)$, then this is clear,
since, near $t$, $p$ must stay within one of the three lines.

Now suppose $p(t) = (0,0) = p'(t)$.  Then
\[
  (0,0) = p'(t) = \lim_{h \to 0} \frac{p(t+h)-p(t)}{h} = \lim_{h \to 0} \frac{p(t+h)}{h} .
\]
Therefore, if the limits exist, we have
\[
  (f \circ p)'(t) = \lim_{h \to 0} \frac{f(p(t+h)) - f(p(t))}{h} = \lim_{h \to 0} \frac{f(p(t+h))}{h} .
\]
But $\|f(p(t+h))\| = \|p(t+h)\|$, %
and so the last limit exists and equals $(0,0,0)$, which is $f(p'(t))$.

For general $k$, we have
$(f \circ p)^{(k)}(t) = (f \circ p')^{(k-1)}(t) = \cdots = (f \circ p^{(k)})(t)$, as required.
\end{proof}

\begin{rem}\label{rem:comparison-map}
Given a pointed diffeological space $(X,x)$,
there is a natural morphism $\beta:T_x(X) \ra \hat{T}_x(X)$
which is defined on the generators by
$\beta(p_*(u))([f])=u([f \circ p])$,
where $u$ is in $T_0(U) \cong \hat{T}_0(U)$,
$p:U \ra X$ is a plot of $X$ with $U$ connected,
$0 \in U$ and $p(0)=x$, and $[f]$ is in $G_x(X)$.
It is straightforward to show that $\beta(p_*(u)):G_x(X) \ra \R$ is smooth.
By the definition of push-forwards of tangent vectors on smooth manifolds,
it is clear that $\beta$ is well-defined.
Then we can linearly extend it to be defined on $T_x(X)$.
Clearly $\beta$ is linear,
and $\beta$ induces a natural transformation $T \ra \hat{T}:\Diff_* \ra \Vect$.
However in general, $\beta$ is neither injective nor surjective; see the following examples.
\end{rem}

\begin{ex}\label{ex:wirediffeology}
(1)
Let $X$ be $\R^n$ equipped with the diffeology generated by the set
$C^\infty(\R,\R^n)$ of smooth curves.
This is sometimes called the ``spaghetti'' or ``wire'' diffeology.
It is clear that $X$ is not diffeomorphic to $\R^n$ with the standard diffeology,
if $n \geq 2$.
By Proposition~\ref{prop:1-plots-determine-external}, we know that
$\hat{T}_x(X) \cong \hat{T}_x(\R^n) \cong \R^n$ for any $x \in X$.

On the other hand, one can show that in this case the relations between
generating internal tangent vectors are determined by smooth curves in $X$,
and hence that $T_x(X)$ has uncountable dimension when $n \geq 2$.
For example, the internal tangent vectors $(p_{\alpha})_*(\frac{d}{dt})$
for $\alpha \in \R$ are all linearly independent, where $p_{\alpha} : \R \to X$
sends $x$ to $(x, \alpha x)$.

This example shows that the internal tangent space of
a diffeological space at a point is not determined by the plots of curves.

(2) Let $Y$ be $\R^n$ equipped with the diffeology generated by the set
$C^\infty(\R^2,\R^n)$ of smooth planes.
(One might call this the ``lasagna'' diffeology.)
Then $Y$ is neither diffeomorphic to $X$ in (1) if $n \geq 2$,
nor diffeomorphic to $\R^n$ with the standard diffeology if $n \geq 3$.
But by Proposition~\ref{prop:2plots} and (1) above,
$T_y(Y) \cong \R^n \cong \hat{T}_y(Y)$ for any $y \in Y$.
\end{ex}

\begin{ex}\label{irrtorus}
Let $T^2=\R^2/\Z^2$ be the usual $2$-torus,
and let $\R_\theta$ be the image of the line $\{ y = \theta x \}$
under the quotient map $\R^2 \ra T^2$, with $\theta$ a fixed irrational number.
Note that $T^2$ is an abelian Lie group,
and $\R_\theta$ is a subgroup.
The quotient group $T^2/\R_\theta$ with the quotient diffeology
is called the \dfn{$1$-dimensional irrational torus of slope $\theta$}.
One can show that $T^2/\R_\theta$ is diffeomorphic to $T^2_{\theta} := \R/(\Z+\theta \Z)$
with the quotient diffeology; see~\cite[Exercise~31 on page~31]{I3}.

Let $x$ be the identity element in $T^2_\theta$.
Since the $D$-topology on $T^2_\theta$ is indiscrete,
the only smooth maps $T^2_\theta \ra \R$ are the constant maps,
which implies that $\hat{T}_x(T^2_{\theta})=0$.
This was also observed in~\cite[Example~4.2]{He}.

On the other hand, we claim that $T_x(T^2_{\theta}) = \R$.
Given a commutative solid square
\[
  \xymatrix@R13pt@C13pt{& U \ar[dl]_f \ar[dr]^g \\ \R \ar[dr]_\pi && \R \ar[dl]^\pi \\ & T^2_\theta}
\]
with $U$ a connected open subset of $\R^n$, $f$ and $g$ smooth maps, and $\pi$ the quotient map,
the difference $f - g$ is a continuous function landing in $\Z + \theta \Z$, and thus is constant.
In particular, if $0 \in U$ and $f(0) = 0 = g(0)$, then $f = g$.
Moreover, since plots $U \to T^2_{\theta}$ locally lift to $\R$, they lift as germs.
(In fact, the uniqueness implies that they lift globally.)
This shows that $\pi$ is a terminal object of $\cG(T^2_{\theta},x)$
and therefore that $T_x(T^2_{\theta}) \cong T_0(\R) \cong \R$.
See~\cite[Example~4.2]{He} for a geometric explanation of this result.
\end{ex}

\begin{ex}\label{halfline:quotient}
Let $O(n)$ act on $\R^n$ in the usual way.
Then the orbit space $H_n$ with the quotient diffeology bijects naturally with
the half line $[0,\infty)$,
but $H_n$ and $[0,\infty)$ with the sub-diffeology of $\R$ are not diffeomorphic,
nor are $H_n$ and $H_m$ for $n \neq m$ (see~\cite{I2}).
By Proposition~\ref{lgcurve}, it is easy to see that $T_{[0]}(H_n)=0$.

We claim that $\hat{T}_{[0]}(H_n)=\R$.
To prove this, observe that the natural map
\[ C^\infty(H_n, \R) \lra \{ f \in C^\infty(\R^n,\R) \mid f \text{ is $O(n)$-invariant} \} \]
is a diffeomorphism, where the right-hand-side is equipped with the sub-diffeology of $C^\infty(\R^n,\R)$.
This follows from the fact that the functor $U \times -$ is a left adjoint
for any open subset $U$ of a Euclidean space, and
therefore commutes with quotients.
Next, consider the natural smooth map
\[ \phi : \{ f \in C^\infty(\R^n, \R) \mid f \text{ is $O(n)$-invariant} \} \lra
          \{ g \in C^\infty(\R,   \R) \mid g \text{ is even} \} , \]
which sends $f$ to the function $g : \R \to \R$ defined by $g(t) = f(t,0,\ldots,0)$,
where the right-hand-side is equipped with the sub-diffeology of $C^\infty(\R,\R)$.
The map $\phi$ is injective, by rotational invariance.
By~\cite[Theorem~1]{Wh}, any smooth even function $g : \R \to \R$ can be
written as $h(t^2)$ for a (non-unique) smooth function $h : \R \to \R$,
which implies that $\phi$ is surjective: given $g$, define $f$ by $f(x) = h(\|x\|^2)$.
To see that $\phi$ is a diffeomorphism, use the Remark at the end of~\cite{Wh},
which says that Theorem~1 holds for functions of several variables which are
even in one of them.
Note that under the identifications provided by the above maps,
the diffeological algebra $G_{[0]}(H_n)$ of germs coincides
with the diffeological algebra $G$ of germs of even smooth functions from $\R$ to $\R$ at $0$.
Thus it suffices to compute the smooth derivations on $G$.

It is easy to check that the map $D : G \to \R$ which sends $g$ to $g''(0)$ is
a derivation.
It is smooth since the second derivative operator~\cite[Lemma~4.3]{CSW} and
evaluation at $0$ are.
And it is non-zero, since $D(j) = 2$,
where $j: \R \ra \R$ is defined by $j(x) = x^2$.
Now suppose that $F$ is any derivation on $G$.
If $g$ is a germ of an even smooth function $\R \to \R$ near $0$,
then $g(x) = g(0) + x^2 h(x)$ for some even smooth germ $h$.
Thus $F(g) = F(j) h(0)$ while $D(g) = 2 h(0)$, and so
$F = (F(j)/2) D$.
Therefore, $\hat{T}_{[0]}(H_n)$ is $1$-dimensional.

We note in passing that $G$ can also be described as the algebra of germs
of smooth functions on $H_1$, so the argument above can be viewed as a
reduction to the case where $n = 1$.
\end{ex}

The technique used in the following example to show that an internal tangent vector is zero
is used in many other examples in this paper.

\begin{ex}\label{halfline:sub}
Let $H_{sub}=[0,\infty)$, as a diffeological subspace of $\R$.
Let us calculate $T_0(H_{sub})$ and $\hat{T}_0(H_{sub})$.

We will first show that $T_0(H_{sub}) = 0$.
By Proposition~\ref{lgcurve}, it suffices to show that every
tangent vector of the form $p_*(\frac{d}{dt})$ is zero, where
$p : \R \to H_{sub}$ is a smooth curve with $p(0) = 0$ and $\frac{d}{dt}$ is the
standard unit vector in $T_0(\R)$.
It follows from Taylor's formula that $p'(0) = 0$,
and hence that $p(x) = x^2 r(x)$ for some plot $r: \R \ra H_{sub}$.
We define $q:\R^2 \ra H_{sub}$ by $q(x,y)=y^2 r(x)$,
which is clearly a plot of $H_{sub}$.
This restricts to $p$ on the diagonal $y=x$, and so
$p_*(\frac{d}{dt}) = q_*(\frac{d}{dx}+\frac{d}{dy})
= q_*(\frac{d}{dx}) + q_*(\frac{d}{dy})$.
Thus our tangent vector is the sum of the tangent vectors
obtained by restricting $q$ to the axes.
The restriction of $q$ to the $x$-axis gives the zero function,
and the restriction of $q$ to the $y$-axis gives
the function $h:\R \ra H_{sub}$ sending $y$ to $y^2 r(0)$.
The former clearly gives the zero tangent vector.
For the latter,
since $h(y) = h(-y)$, we have that $h_*(\frac{d}{dy}) = h_*(-\frac{d}{dy}) = -h_*(\frac{d}{dy})$,
which implies that $h_{*}(\frac{d}{dy}) = 0$.
Hence our original vector is $0$, and so $T_0(H_{sub})=0$.

To calculate $\hat{T}_0(H_{sub})$,
consider the squaring map $\alpha : H_1 \to H_{sub}$, where $H_1 = \R/O(1)$ is
described in Example~\ref{halfline:quotient}.
The map $\alpha$ is smooth, since it fits into a diagram
\[
  \xymatrix{\R \ar[r]^j \ar@{->>}[d] & \R \\
            H_1 \ar[r]_-{\alpha}     &  H_{sub}\,.\, \ar@{ (->}[u] }
\]
We therefore have a smooth map $\alpha^* : C^{\infty}(H_{sub}, \R) \to C^{\infty}(H_1, \R)$.
Regarding the latter as consisting of even smooth functions $\R \to \R$,
$\alpha^*$ sends a smooth function $h : H_{sub} \to \R$ to the smooth function
mapping $t$ to $h(t^2)$.
This is clearly injective, and by~\cite{Wh} it is surjective and a diffeomorphism.
Also note that by \cite[Lemma~3.17]{CSW}, the $D$-topology on $H_{sub}$ is the same as the sub-topology of $\R$.
This implies that the algebras $G_0(H_{sub})$ and $G_{[0]}(H_1)$ of germs are
isomorphic as diffeological algebras, and therefore that
$\hat{T}_0(H_{sub}) \cong \hat{T}_{[0]}(H_1) \cong \R$.
A non-zero element of $\hat{T}_0(H_{sub})$ is given by the smooth derivation sending $f$ to $f'(0)$.
Under the isomorphism, this corresponds to $D/2 \in \hat{T}_{[0]}(H_1)$
from Example~\ref{halfline:quotient}.
\end{ex}

\subsection{Other approaches to tangent spaces}\label{ss:other-approaches}

We have studied the internal and external tangent spaces in some detail
because we find them to be the most natural definitions.
Rather than arguing that one of them is the ``right'' definition, we
simply point out that it depends on the application.
The internal tangent space fits most closely with the definition of a
diffeological space, as it works directly with the plots, so it will
form the basis for the next section, on tangent bundles.

There are many other possible approaches to defining tangent spaces,
and again, these will be useful for different applications.
We briefly summarize a few variants of our approaches here.

\subsubsection{Variants of the internal tangent space}

We have seen in Proposition~\ref{lgcurve} that the internal tangent space
$T_x(X)$ is spanned by vectors of the form $p_*(\frac{d}{dt})$,
where $p : \R \to X$ is a smooth curve sending $0$ to $x$.
One could instead use smooth maps from $[0, \infty)$ to $X$ sending
$0$ to $x$, with relations coming from higher-dimensional quadrants.
With such a definition, $[0, \infty)$ as a diffeological subspace of $\R$
would have a non-trivial tangent space at $0$.

Our internal tangent spaces were defined as colimits in the category of vector spaces.
One could instead take the colimit in the category of sets, which would mean
that every tangent vector is representable by a smooth curve.

These two choices are independent;  one could form a tangent space
using either one or both of them.

\subsubsection{Variants of the external tangent space}

An external tangent vector is defined to be a smooth derivation $G_x(X) \to \R$.
One might instead consider all derivations.
For all examples in Subsection~\ref{ss:examples} we would get the same external tangent spaces,
but in general we don't know if every derivation is automatically smooth.

Independently, one could also change the definition of the space $G_x(X)$
of germs of smooth functions at $x$, for a general diffeological space $X$.
It was defined at the beginning of Subsection~\ref{ss:external} using the
$D$-topology on $X$.
However, when studying diffeological bundles~\cite{I1}, the correct notion
of ``locality'' uses the plots, not the $D$-open subsets.
Moreover, for global smooth functions, we have that
$C^{\infty}(X, \R) = \lim C^{\infty}(U, \R)$,
where the limit is taken in the category of sets and the indexing category
is the category $\DS/X$ of plots $U \to X$.
Thus, by analogy, one might define the germs of smooth functions at $x$ to be
$\lim G_0(U)$,
where the indexing category is the category $\DS_0/(X,x)$ of pointed plots
$(U,0) \to (X,x)$.
This has the advantage of being closer to the definition of the internal tangent
space, but we have not investigated it in detail.

There is an established definition of differential forms on a
diffeological space~\cite{I3}, and one could instead define the
tangent space to be the dual of the space of $1$-forms at a point.
This is closely related to the external tangent space we study,
since a germ $f$ of a smooth function gives rise to a $1$-form $df$.

\subsubsection{A mixed variant}

Finally, recall that in Remark~\ref{rem:comparison-map} we defined a
natural transformation $\beta : T \to \hat{T}$.
One could consider the image of this natural transformation.
That is, one identifies internal tangent vectors if they give rise
to the same directional derivative operators.
Something very close to this is done in~\cite{V}, except that
the derivative operators are only compared on global smooth functions
rather than germs of smooth functions.

In~\cite{I3}, a mixed variant of the definition involving
$1$-forms is proposed, and worked out in some detail.
Each $1$-dimensional plot centered at a point gives rise to a smooth
linear functional on $1$-forms, and the tangent space is defined to
be the span of such linear functionals.
This definition agrees with the internal tangent space in all examples
where we know both.

One can also consider mixed variants involving sets rather than
vector spaces.  For example, the set of external tangent
vectors which are represented by smooth curves is called the
kinematic tangent space in~\cite{St2}.
A very similar approach was taken in~\cite{So1}, which discussed
the case in which $X$ is a diffeological group, and used certain ``states''
$X \to \C$ to determine when two internal tangent vectors are equal.

\section{Internal tangent bundles}\label{s:bundle}

We discussed the internal tangent space of a diffeological space at a point in Subsection~\ref{ss:internal}.
As usual, if we gather all of the internal tangent spaces together, we can form the internal tangent bundle.
In Subsection~\ref{ss:dvs}, we begin by recalling the diffeology Hector defined on this bundle~\cite{He}.
We show that it is not well-behaved in general
and also point out some errors in~\cite{He,HM,La}.
We then introduce a refinement of Hector's diffeology, called the dvs diffeology,
give several examples and counterexamples, and describe the internal tangent bundle
of a diffeological group.
In Subsection~\ref{ss:conceptual}, we give a conceptual explanation of the
relationship between Hector's diffeology and the dvs diffeology:  they are
colimits, taken in different categories, of the same diagram.
In Subsection~\ref{ss:fine}, we study the question of when internal tangent spaces
are fine diffeological vector spaces.
Finally, in Subsection~\ref{ss:function}, we study the internal tangent bundles of function spaces,
and generalize a result in~\cite{He,HM} that says that the internal tangent space of
the diffeomorphism group of a compact smooth manifold at the identity is isomorphic
to the vector space of all smooth vector fields on the manifold.

\subsection{Definitions and examples}\label{ss:dvs}

In this subsection, we recall Hector's diffeology on the internal tangent bundle
of a diffeological space.  Then we observe in Example~\ref{ex:bundleofcross} that
when the internal tangent bundle is equipped with Hector's diffeology,
neither addition nor scalar multiplication are smooth in general.
This provides a counterexample to some claims in~\cite{He,HM,La}.
To overcome this problem, we introduce a refinement of Hector's diffeology,
called the dvs diffeology, on the tangent bundle.
Proposition~\ref{prop:subandproductforbundles} extends some results
about internal tangent spaces to internal tangent bundles.
Then we prove in Theorem~\ref{thm:Tdiffgrptrivial}
that the internal tangent bundle of a diffeological group is always trivial,
as the original proof in~\cite{HM} was partially based on a false result.
We also show that, in this case, Hector's diffeology and the dvs diffeology agree.

\begin{de}
The \dfn{internal tangent bundle} $TX$ of a diffeological space $X$
is defined to be the set $\coprod_{x \in X} T_x(X)$.
\dfn{Hector's diffeology}~\cite{He} on $TX$ is generated by the maps $Tf:TU \ra TX$,
where $f:U \ra X$ is a plot of $X$ with $U$ connected, $TU$ has the standard diffeology,
and for each $u \in U$, $T_u f: T_u(U) \ra T_{f(u)}(X)$ is defined to be the composite
$T_u(U) \ra T_0(U-u) \ra T_{f(u)}(X)$, with $U-u$ the translation of $U$ by $u$.
The internal tangent bundle of $X$ with Hector's diffeology is denoted $T^H(X)$,
and $T_x(X)$ with the sub-diffeology of $T^H(X)$ is denoted $T_x^H(X)$.
We write elements of $TX$ as $(x, v)$, where $v \in T_x(X)$.
\end{de}

Recall that by the universal property of colimits,
for any smooth map $f:X \ra Y$ and any $x \in X$,
we have a linear map $f_*: T_x(X) \ra T_{f(x)}(Y)$.
It is straightforward to check that $T^H:\Diff \ra \Diff$ is a functor,
and hence $f_*: T_x^H(X) \ra T_{f(x)}^H(Y)$ is smooth.
Moreover, the natural map $\pi_X:T^H(X) \ra X$ is smooth (indeed, it is a subduction),
and hence $\pi:T^H \ra 1$ is a natural transformation.
Also the zero section $X \ra T^H(X)$ is smooth.

\begin{de}\label{def:dvs}
A \dfn{diffeological vector space} is a vector space object in $\Diff$.
More precisely, it is both a diffeological space and a vector space
such that the addition and scalar multiplication maps are smooth.
\end{de}

The following example shows that $T_x^H(X)$ is not a diffeological vector space
in general.
In fact, both the addition map $T_x^H(X) \times T_x^H(X) \ra T_x^H(X)$
and the scalar multiplication map $\R \times T_x^H(X) \ra T_x^H(X)$
can fail to be smooth.
Therefore, both~\cite[Proposition~6.6]{HM} and~\cite[Lemma~5.7]{La} are false.

\begin{ex}\label{ex:bundleofcross}
Let $X$ be the diffeological space introduced in Example~\ref{ex:cross-pushout},
two copies of $\R$ glued at the origin thought of as a subset of $\R^2$.
We first show that the addition map
$T_{(0,0)}^H(X) \times T_{(0,0)}^H(X) \ra T_{(0,0)}^H(X)$ is not smooth.
To see this, let $f, g: \R \ra T_{(0,0)}(X)$ be given by
$f(x)=(x,0)$ and $g(x)=(0,x)$,
where we use the natural identification $T_{(0,0)}(X) \cong \R^2$.
Clearly $f$ and $g$ are smooth as maps $\R \ra T_{(0,0)}^H(X)$.
We will show that the sum $h = f + g: \R \ra T_{(0,0)}^H(X)$
given by $x \mapsto (x,x)$ is not smooth.
Any plot $p: U \to X$ must locally factor through one of the axes,
and so the diffeology on $T_{(0,0)}^H(X)$ is generated by plots
factoring through one of the axes.
Clearly $h$ is not locally constant and
does not locally factor through a generating plot,
so it is not smooth.
Now we show that scalar multiplication
$\R \times T_{(0,0)}^H(X) \ra T_{(0,0)}^H(X)$ is not smooth.
This is because the (non-smooth) map $h$ can be described as the composite
$\R \ra \R \times T_{(0,0)}^H(X) \ra T_{(0,0)}^H(X)$,
where the first map is given by $t \mapsto (t,(1,1))$, which is clearly smooth.
Note that, away from the axes, the diffeology on $T_{(0,0)}^H(X)$ is discrete.
\end{ex}

\begin{rem}\label{rem:scalar-mult}
It follows that the fibrewise scalar multiplication map $\R \times T^H(X) \to T^H(X)$
is also not smooth in general.
This is particularly surprising given the following observation.
A generating plot of $T^H(X)$ is of the form $Tf : TU \to T^H(X)$,
where $f : U \to X$ is a plot of $X$ with $U$ connected.
Thus one might expect the plots $1_{\R} \times Tf : \R \times TU \to \R \times T^H(X)$
to generate the product diffeology on $\R \times T^H(X)$.
If so, the commutative square
\[
\xymatrix@C+10pt{\R \times TU \ar[r]^-{1_{\R} \times Tf} \ar[d] & \R \times T^H(X) \ar[d] \\
                    TU \ar[r]_{Tf}                   & T^H(X) }
\]
would imply that scalar multiplication on $T^H(X)$ is smooth.
The problem is that there may be tangent vectors in $T^H(X)$ which are
not in the image of any generating plot $TU \to T^H(X)$, in which case
one needs to consider constant plots as well.
However, this argument has shown that if every tangent vector in $X$ is
$1$-representable (Remarks~\ref{rem:infinite-product} and~\ref{rem:representable}),
then scalar multiplication is smooth.
This will be the case when $X$ is a diffeological group (Remark~\ref{rem:diff-group}).
It is not hard to see that the converse is true as well:  when scalar multiplication
is smooth, every tangent vector in $X$ is $1$-representable.
\end{rem}

We will introduce a new diffeology on the internal tangent bundle of a diffeological space
that makes the addition and scalar multiplication maps smooth.
We first introduce the following concept:

\begin{de}
Let $X$ be a diffeological space.
A \dfn{diffeological vector space over $X$} is a diffeological space $V$,
a smooth map $p : V \to X$
and a vector space structure on each of the
fibres $p^{-1}(x)$ such that the addition map $V \times_X V \to V$,
the scalar multiplication map $\R \times V \to V$
and the zero section $X \ra V$ are smooth.
Here $\R \times V$ has the product diffeology
and $V \times_X V$ has the sub-diffeology of the product diffeology on $V \times V$.
In other words, a diffeological vector space over $X$ is
a vector space object in the overcategory $\Diff/X$
over the field object $\pr_2: \R \times X \to X$.

In the case when $X$ is a point, we recover the concept of diffeological vector space.
\end{de}

Note that if $V$ is a diffeological vector space over $X$, then it
is automatically the case that each fibre of $p$, with the sub-diffeology,
is a diffeological vector space.
It also follows that $p$ is a subduction.

\medskip

Now we give a construction that we will use to enlarge the
diffeology on $T^H(X)$ in order to make it a diffeological
vector space over $X$.
A similar construction can be found in~\cite[Theorem~5.1.6 and Definition~6.2.1]{V},
using a different definition of the tangent spaces.
One can show that the notion of \emph{regular vector bundle} in~\cite{V}
exactly matches our notion of a diffeological vector space over a
diffeological space.

\begin{prop}\label{prop:dvs-over-X}
Let $p:V \to X$ be a smooth map between diffeological spaces, and
suppose that each fibre of $p$ has a vector space structure.  Then there
is a smallest diffeology $\cD$ on $V$ which contains the given diffeology
and which makes $V$ into a diffeological vector space over $X$.
\end{prop}

\begin{proof}
We first take the largest diffeology on $V$ making $p: V \to X$ smooth.
It is easy to see that this diffeology contains the original diffeology on $V$
and makes $V$ into a diffeological vector space over $X$.
Now consider the intersection $\cD$ of all diffeologies $\cD_i$ on $V$ which
have these two properties.
We claim that $\cD$ also has these two properties.
It is clear that $\cD$ contains the original diffeology on $V$
and that $p:(V,\cD) \ra X$ and the zero section $X \ra (V,\cD)$ are smooth.
A plot in $V \times_X V$ consists of plots $q_1, q_2 : U \to V$ in $\cD$ such that
$p \circ q_1 = p \circ q_2$.
Since the pointwise sum $q_1+q_2 : U \to V$ is in each of
the diffeologies $\cD_i$, it is in $\cD$ as well.
Thus $V \times_X V \to V$ is smooth with respect to $\cD$.
Similarly, $\R \times V \to V$ is smooth.
Therefore, $(V,\cD)$ is a diffeological vector space over $X$.
\end{proof}

We write $\tilde{V}$ for $V$ equipped with the diffeology $\cD$.

In the special case where $X$ is a point and the vector space $V$ starts with the
discrete diffeology, the above proposition proves that there is a smallest
diffeology on $V$ making it into a diffeological vector space.
This diffeology is called the fine diffeology;
see Subsection~\ref{ss:fine} and~\cite[Chapter~3]{I3}.

Note that the largest diffeology on $V$ making $p: V \to X$ smooth, which
was used in the proof of the above proposition, is not
generally interesting, since it induces the indiscrete diffeology on each fibre.

\begin{rem}\label{rem:alt-description}
One can give a more concrete description of the diffeology $\cD$ on $V$
described in Proposition~\ref{prop:dvs-over-X}:
it is generated by the linear combinations of the original plots of $V$
and the composite of the zero section with plots of $X$.
More precisely, given a plot $q : U \to X$,
plots $q_1, q_2, \ldots, q_k : U \to V$
such that $p \circ q_i = q$ for all $i$,
and plots $r_1, r_2, \ldots, r_k: U \to \R$ in the standard diffeology on $\R$,
the linear combination $U \to V$ sending $u$ to $r_1(u) q_1(u) + \cdots + r_k(u) q_k(u)$
in $p^{-1}(q(u))$ is a plot in $\cD$, and every plot in $\cD$ is locally of this form.
Note that when $k=0$, this is the composite of the plot $q$ of $X$
with the zero section of $V$.

One consequence is the following description of the fibres of $\tilde{V}$.
For $x \in X$, write $\tilde{V}_x$ for $p^{-1}(x)$ with the sub-diffeology of $\tilde{V}$
and $\widetilde{V_x}$ for the same set with the diffeology obtained by
starting with the sub-diffeology of $V$ and completing it to a vector space
diffeology using Proposition~\ref{prop:dvs-over-X}.
It is not hard to see that these diffeologies agree.
(See also~\cite[Proposition~6.2.2(iiii)]{V}.)
More generally, this construction commutes with pullbacks.
\end{rem}

The following two results also follow immediately from Remark~\ref{rem:alt-description}.

\begin{prop}\label{prop:Hectortodvs}
Let
\[
\xymatrix{V \ar[r]^f \ar[d]_p & W \ar[d]^q \\ X \ar[r]_g & Y}
\]
be a commutative diagram in $\Diff$, such that each fibre of $p$ and $q$
has a vector space structure and $f|_{p^{-1}(x)}:p^{-1}(x) \ra q^{-1}(g(x))$
is linear for each $x \in X$.
Then the map $f: \widetilde{V} \ra \widetilde{W}$ is smooth.
Furthermore, if both $f$ and $g$ in the original square are inductions,
then so is $f: \widetilde{V} \ra \widetilde{W}$.
\end{prop}

\begin{prop}\label{pr:dvs-product}
Let $p:V \ra X$ and $q:W \ra Y$ be smooth maps between diffeological spaces.
Assume that each fibre of $p$ and $q$ has a vector space structure.
Then $\widetilde{V} \times \widetilde{W}$ is isomorphic to
$\widetilde{V \times W}$ as diffeological vector spaces over $X \times Y$.
\end{prop}

Here is the new diffeology on the internal tangent bundle of a diffeological space:

\begin{de}\label{def:dvsdiffeology}
We define the \dfn{dvs diffeology} on $TX$ to be the smallest diffeology containing
Hector's diffeology which makes $TX$ into a diffeological vector space over $X$.
The internal tangent bundle of $X$ with the dvs diffeology is denoted $T^{dvs}(X)$,
and $T_x(X)$ with the sub-diffeology of $T^{dvs}(X)$ is denoted $T_x^{dvs}(X)$.
\end{de}

As a convention, whenever Hector's diffeology coincides with the dvs diffeology
for a tangent space or a tangent bundle, we omit the superscript.

\smallskip

By Proposition~\ref{prop:Hectortodvs} and
the corresponding results for Hector's diffeology,
it is clear that $T^{dvs}:\Diff \ra \Diff$ is a functor,
and we have a natural transformation $\pi:T^{dvs} \ra 1$.
Hence, for any smooth map $f:X \ra Y$ and any $x \in X$,
the induced map $f_*: T_x^{dvs}(X) \ra T_{f(x)}^{dvs}(Y)$
is a smooth linear map between diffeological vector spaces.
Also, the natural map $\pi_X:T^{dvs}(X) \ra X$ is a subduction.

\begin{ex}
When $X$ is a smooth manifold, it follows from Example~\ref{smoothmanifolds} that
$TX$ agrees with the usual tangent bundle as a set.
In fact, it is not hard to check that
Hector's diffeology on $TX$ coincides with the standard diffeology,
regarding $TX$ as a smooth manifold.
Since $TX$ is a diffeological vector space over $X$,
the dvs diffeology on $TX$ also coincides with the standard diffeology.
\end{ex}

\begin{rem}\label{re:bundle}
Example~\ref{ex:cross-pushout} shows that $TX \ra X$ is not
a diffeological bundle~\cite{I1} in general,
whether $TX$ is equipped with Hector's diffeology or the dvs diffeology,
since the pullback along a non-constant plot passing through the origin
would have fibres of different dimensions.
The same example also shows that $TX \ra X$ is not a fibration
in the sense of~\cite{CW}, with either diffeology.
Indeed, there is no dashed arrow in $\Diff$ making the diagram
\[
  \xymatrix{
    0 \ar[r] \ar@{ (->}[d] & TX \ar[d] \\
   \R \ar[r] \ar@{-->}[ur] & X
  }
\]
commute, where the top map sends $0$ to $(1,1) \in T_{(0,0)}(X)$
and the bottom map is the inclusion of the $x$-axis.
\end{rem}

Now we extend Propositions~\ref{prop:internallocal} and~\ref{product}
from internal tangent spaces to internal tangent bundles:

\begin{prop}\label{prop:subandproductforbundles}\
\begin{enumerate}
\item \label{subforbundles} If $A$ is a $D$-open subset of a diffeological space $X$,
equipped with the sub-diffeology of $X$,
then $T^H(A) \ra T^H(X)$ is an induction such that $T_a^H(A) \ra T_a^H(X)$
is a diffeomorphism for each $a \in A$.
The same is true for the dvs diffeology.

\item \label{productforbundles} Let $X$ and $Y$ be diffeological spaces.
Then there is a natural diffeomorphism $T^H(X \times Y) \ra T^H(X) \times T^H(Y)$
which commutes with the projections to $X \times Y$,
and the same is true for the dvs diffeology.
\end{enumerate}
\end{prop}
\begin{proof}
(1)
To see that $T^H(A) \ra T^H(X)$ is an induction, note
that for any plot $p:U \ra X$, $p^{-1}(A)$ is an open subset
of $U$ since $A$ is $D$-open in $X$,
and we have a commutative square
\[
\xymatrix{T(p^{-1}(A)) \ar@{ (->}[r] \ar[d] & TU \ar[d]^{Tp} \\ T^H(A) \ar[r] & T^H(X).}
\]
It follows that, for each $a$ in $A$, the map $T_a^H(A) \ra T_a^H(X)$ is an induction as well.
By Proposition~\ref{prop:internallocal}, this map is an isomorphism
of vector spaces, and therefore it is a diffeomorphism.
The corresponding result for the dvs diffeology then essentially follows from Proposition~\ref{prop:Hectortodvs}.

(2) The first statement follows directly from Proposition~\ref{product}
and the fact that for any plots $p:U \ra X$ and $q:V \ra Y$, we have
the following commutative square
\[
\xymatrix{T(U \times V) \ar[d]_{T(p \times q)} \ar[r]^-{\cong} & TU \times TV \ar[d]^{Tp \times Tq} \\ T^H(X \times Y) \ar[r] & T^H(X) \times T^H(Y),}
\]
and the second statement follows then from Proposition~\ref{pr:dvs-product}.
\end{proof}

We end this subsection with a discussion of the internal tangent bundles of diffeological groups
(see Definition~\ref{def:diff-group}).

\begin{rem}\label{rem:diff-group}\
\begin{enumerate}
\item Let $e$ be the identity in a diffeological group $G$.
It is shown in~\cite[Proposition~6.4]{HM} that the multiplication map $G \times G \ra G$
induces the addition map $T_e(G) \times T_e(G) \ra T_e(G)$.
It follows that the addition map
$T_e^H(G) \times T_e^H(G) \ra T_e^H(G)$ is smooth.
Moreover,~\cite[Corollary~6.5]{HM} implies that every element in $T_e(G)$
can be written as $p_*(\frac{d}{dt})$, where $p:\R \ra G$ is a single smooth curve with $p(0)=e$,
and so it follows from Remark~\ref{rem:scalar-mult} that scalar multiplication is smooth as well.
Therefore, $T_e^H(G)$ is a diffeological vector space.
Since left-multiplication by any $g \in G$ is a diffeomorphism,
we also see that $T_g^H(G)$ is a diffeological vector space
and that $T_e^H(G)$ and $T_g^H(G)$ are isomorphic as diffeological vector spaces.
Similarly, $T_e^{dvs}(G)$ and $T_g^{dvs}(G)$ are isomorphic as diffeological vector spaces.

\item If $V$ is a diffeological vector space, one might hope that the
scalar multiplication map $\R \times V \to V$ induces the scalar multiplication
map $\R \times T_0(V) \to T_0(V)$, where we use that $T_r(\R) \cong \R$ for any $r \in \R$
and that $T$ respects finite products.
(We omit superscripts here because we are not using the diffeology in this remark.)
But since induced maps are always linear, and scalar multiplication is bilinear,
this will only happen when $T_0(V) = 0$.
Instead, for fixed $r \in \R$, we can consider the map $V \to V$ sending $v$ to $r v$.
It is not hard to show that for $r \in \Q$, this induces the map $T_0(V) \to T_0(V)$
sending $u$ to $r u$.
We conjecture that this is true for any $r \in \R$.
It is true if for any $u \in T_0(V)$, there exist plots $p: \R \to V$ and 
$q: \R \to V$ such that $p(0) = 0$, $p_*(\ddt) = u$ and $p(t) = t q(t)$ for $t$ in a neighborhood of $0 \in \R$. 
This condition is satisfied when $V = C^\infty(X, \R^n)$ with the functional diffeology, for any 
diffeological space $X$, or when $V$ is a retract of such a space in the category of diffeological vector spaces.
\end{enumerate}
\end{rem}

Although the proof of~\cite[Proposition~6.8]{HM} was partly based on a false result,
the proposition is still correct:

\begin{thm}\label{thm:Tdiffgrptrivial}
Let $G$ be a diffeological group.
Then $T^H(G)$ is a diffeological vector space over $G$ and
all of $T^H(G)$, $G \times T_e^H(G)$, $T^{dvs}(G)$ and $G \times T_e^{dvs}(G)$
are isomorphic as diffeological vector spaces over $G$.
Therefore, $T_g^H(G) = T_g^{dvs}(G)$ for any $g \in G$.
\end{thm}

\begin{proof}
In the proof of~\cite[Proposition~6.8]{HM}, Hector and Macias-Virgos considered the map
$F:G \times T_e^H(G) \ra T^H(G)$ given by $(g,v) \mapsto (g,(L_g)_*(v))$,
where $L_g:G \ra G$ is left multiplication by $g$,
and they argued that $F$ is a diffeomorphism.
But in the proof that $F$ is smooth, they used that
$T_e^H(G)$ is a fine diffeological vector space, which is not true in general;
see Example~\ref{ex:notfine} for counterexamples.
It is easy to fix this.
Let $a$ be the composite $G \times T^H(G) \ra T^H(G) \times T^H(G)
\cong T^H(G \times G) \ra T^H(G)$, where the first map is given by
$\sigma \times 1_{T^H(G)}$ with $\sigma:G \ra T^H(G)$ the zero section,
the second map is the isomorphism from Proposition~\ref{prop:subandproductforbundles}(2),
and the third map is induced from the multiplication $G \times G \ra G$.
Clearly $a$ is smooth, and it is straightforward to check that $F$ equals $a|_{G \times T_e^H(G)}$,
and is therefore smooth.
One can also see that $F^{-1} : T^H(G) \to G \times T_e^H(G)$ is given by
$F^{-1}(g, v) = (g, a(g^{-1}, (g, v)))$, which is smooth.
So $T^H(G)$ is diffeomorphic to $G \times T_e^H(G)$, and the diffeomorphism
respects the projections to $G$ and the linear structures on the fibres.
In particular, $T^H(G)$ is a diffeological vector space over $G$.
The rest then follows directly from Definition~\ref{def:dvsdiffeology}.
(The last statement also follows from the second paragraph of 
Remark~\ref{rem:alt-description} and Remark~\ref{rem:diff-group}(1).)
\end{proof}

Because of the previous result, our convention allows us to write $TG$ and $T_g(G)$
(without superscripts) when $G$ is a diffeological group.
This includes the case when $G$ is a diffeological vector space.

\subsection{A conceptual description of the Hector and dvs diffeologies}\label{ss:conceptual}

In this subsection we give a categorical explanation of the difference between Hector's diffeology
and the dvs diffeology on the internal tangent bundle of a diffeological space.
To summarize briefly, they can both be described as the colimit of a natural diagram,
but the colimits take place in different categories.
This material is not needed in the rest of the paper.

\medskip

For a diffeological space $X$,
a \dfn{vector space with diffeology over $X$}
is a diffeological space $V$, a smooth map $p:V \ra X$
and a vector space structure on each of the fibres $V_x = p^{-1}(x)$,
with no compatibility conditions.
The category $\VSD$ has as objects the vector spaces with diffeology
over diffeological spaces, and as morphisms the commutative squares
\[
\xymatrix{V \ar[r]^g \ar[d] & W \ar[d] \\ X \ar[r]_f & Y}
\]
in $\Diff$ such that
for each $x \in X$, $g|_{V_x}: V_x \ra W_{f(x)}$ is linear.

\begin{prop}
The category $\VSD$ is complete and cocomplete.
\end{prop}

\begin{proof}
Let $F: I \ra \VSD$ be a functor from a small category,
and write $V_i \ra X_i$ for $F(i)$.
There are two functors $t,b:\VSD \ra \Diff$,
sending the above commutative square to $g:V \ra W$ and $f:X \ra Y$, respectively.
Write $\lim V_i$ and $\lim X_i$ for the limits of the functors
$t \circ F$ and $b \circ F$, respectively.
Then it is easy to check that there is a canonical smooth map
$\lim V_i \ra \lim X_i$ which is the limit of $F$ in $\VSD$.

Next, write $X=\colim X_i$ for the colimit of the functor $b \circ F$.
Let $\cC$ be the category of elements of $b \circ F$.
Then $\cC$ has as objects the pairs $(i,a)$ for $i \in \obj(I)$ and $a \in X_i$,
and the morphisms $(i,a) \ra (j,b)$ in $\cC$ are the morphisms $f:i \ra j$ in $I$
such that $b(F(f)): X_i \ra X_j$ sends $a$ to $b$.
There is a natural bijection between $\pi_0(\cC)$ and the underlying set of $X$.
For any $x \in X$, the connected full subcategory $\cC_x$
of $\cC$ corresponding to $x$ consists of the objects $(i,a)$
such that $a \in X_i$ is sent to $x \in X$ by the cocone map $X_i \to X$.
There is a functor $\cC_x \ra \Vect$ sending $(i,a)$ to
$V_{i,a}$, the fibre above $a \in X_i$ in $V_i$, and
sending $f:(i,a) \ra (j,b)$ to $t(F(f))|_{V_{i,a}}:V_{i,a} \ra V_{j,b}$.
Let $V_x$ be the colimit of this functor
and let $V$ be the disjoint union $\coprod_{x \in X} V_x$.
There is a canonical map $V_i \ra V$ for each $i \in \obj(I)$,
and we equip $V$ with the smallest diffeology making these maps smooth.
Then the canonical projection $V \ra X$ is smooth,
and one can check that this is the colimit of $F$.
\end{proof}

We write $\DVS$ for the full subcategory of $\VSD$ with objects
diffeological vector spaces over diffeological spaces (see Definition~\ref{def:dvs}).
Propositions~\ref{prop:dvs-over-X} and~\ref{prop:Hectortodvs} imply that
the forgetful functor $\DVS \ra \VSD$ has a left adjoint,
sending $p:V \ra X$ to $p:\tilde{V} \ra X$.
It is clear that if $p:V \ra X$ is a diffeological vector space over $X$,
then $\tilde{V}=V$.
Therefore, the category $\DVS$ is also complete and cocomplete,
with the limit computed in $\VSD$,
and the colimit obtained by applying the left adjoint functor
to the corresponding colimit in $\VSD$.
Also, Proposition~\ref{pr:dvs-product} says that
the left adjoint commutes with finite products.

\medskip

Recall that $\DS$ is the category with objects all open subsets of $\R^n$
and morphisms smooth maps between them.
The main result of this subsection is now straightforward to check.

\begin{thm}\label{thm:dvs-vs-Hector}
Let $X$ be a diffeological space.
Consider the functor $\DS/X \ra \DVS$ defined by
\[
\vcenter{\xymatrix@C5pt@R-5pt{U \ar[dr]_-p \ar[rr]^f & & V \ar[dl]^-q \\ & X_{\strut} }}
\quad\longmapsto\quad
\vcenter{\xymatrix{TU \ar[r]^{Tf} \ar[d]_{\pi_U} & TV \ar[d]^{\pi_V} \\ U \ar[r]_f & V.}}
\]
The colimit of this functor is $\pi_X: T^{dvs}(X) \ra X$,
and the colimit of the composite functor $\DS/X \ra \DVS \ra \VSD$
is $\pi_X:T^H(X) \ra X$.
\end{thm}

This can also be phrased as saying that
the functors $T^H:\Diff \ra \VSD$ and $T^{dvs}:\Diff \ra \DVS$ are
left Kan extensions of the natural functors
$\DS \ra \VSD$ and $\DS \ra \DVS$
along the inclusion functor $\DS \to \Diff$.

\subsection{Fineness of internal tangent spaces}\label{ss:fine}

Recall that any vector space $V$ has a smallest diffeology making it into a
diffeological vector space.
This diffeology is called the \dfn{fine} diffeology,
and its plots are exactly those maps which locally are of the
form $U \to \R^n \to V$, where the first map is smooth and the
second map is linear; see~\cite[Chapter~3]{I3}.
In this subsection, we study when $T^H_x(X)$ and $T^{dvs}_x(X)$ are
fine diffeological vector spaces, giving examples where they are
fine and where they are not.

The material in this subsection is independent of the material in
the following subsection.

\medskip

From Example~\ref{ex:bundleofcross},
we know that $T^H_x(X)$ in general is not a fine diffeological vector space.
We now give a condition implying that $T^{dvs}_x(X)$ is fine.

\begin{prop}\label{prop:fine}
Let $(X,x)$ be a pointed diffeological space.
Assume that there exists a local generating set $G$ of $X$ at $x$ with the property
that for any $q:U \ra X$ in $G$,
if $f : V \to U$ is a smooth map from an open subset $V$ of $\R^n$ such that
$q \circ f = c_x$ (the constant map), then $f$ is locally constant.
Then $T_x^{dvs}(X)$ is a fine diffeological vector space.
\end{prop}

\begin{proof}
By the second half of Remark~\ref{rem:alt-description}, the
diffeology on $T^{dvs}_x(X)$ is the smallest diffeology containing
Hector's diffeology which makes $T^{dvs}_x(X)$ into a diffeological
vector space.
Thus it suffices to show that Hector's diffeology is contained
in the fine diffeology.
That is, we need to show that every plot $V \ra T^H_x(X)$
locally smoothly factors through a linear map $\R^n \to T^H_x(X)$.
Let $p : V \to T^H_x(X)$ be a plot.
Since Hector's diffeology on $TX$ is generated by maps $TU \to TX$
induced by plots $U \to X$, for each $v_0 \in V$ there is a connected
open neighbourhood $V'$ of $v_0$ in $V$ such that $p|_{V'}$ is either
constant or of the form $V' \llra{g} TU \to TX$, with image in $T_x(X)$,
where $g$ is smooth.
In the first case, every constant map factors smoothly through a
linear map $\R \to T_x(X)$.
In the second case, shrinking $V'$ further if necessary,
we can assume that $U \to X$ is in the local generating set $G$.
The composite $V' \llra{g} TU \to U \to X$ is constant, and so by hypothesis,
$g$ must land in $T_{u}(U)$ for some $u \in U$.
This shows that $p|_{V'}$ smoothly factors through the linear map $T_u(U) \to T_x(X)$.
Since $T_u(U)$ is diffeomorphic to $\R^n$, for some $n$, we are done.
\end{proof}

Here are some results that follow from Proposition~\ref{prop:fine}:

\begin{ex}\label{ex:fineexamples}
$T_x^{dvs}(X)$ is a fine diffeological vector space when:
\begin{enumerate}
\item $X$ is a smooth manifold, and $x$ is any point of $X$;
\item $X$ is the axes in $\R^2$ with the pushout diffeology (Example~\ref{ex:cross-pushout}),
and $x$ is any point of $X$;
\item\label{tangentirrtorus} $X$ is a $1$-dimensional irrational torus (Example~\ref{irrtorus}),
and $x$ is any point of $X$.
\end{enumerate}
\end{ex}

\begin{rem}\label{rem:inverse}
Since the irrational torus $T^2_\theta$ is a diffeological group, by Theorem~\ref{thm:Tdiffgrptrivial}
we know that $T^H_{[0]}(T^2_\theta) \cong T^{dvs}_{[0]}(T^2_\theta)$.
So by (\ref{tangentirrtorus}) of the above example, both of them have the fine diffeology.
Therefore, the map $T_0(\R) \ra T_{[0]}(T^2_\theta)$ induced by the
quotient map $\R \to T^2_\theta$ from Example~\ref{irrtorus} is a linear diffeomorphism.
As a consequence, the inverse function theorem does not hold for general diffeological spaces
when tangent spaces are equipped with either Hector's diffeology or the dvs diffeology,
in the sense that, if $f:A \ra B$ is a smooth map between diffeological spaces such that
$f_*:T_a^H(A) \ra T_{f(a)}^H(B)$ (or $f_*:T_a^{dvs}(A) \ra T_{f(a)}^{dvs}(B)$) is a linear diffeomorphism for some $a \in A$,
then it is not true that there exist $D$-open neighborhoods $A'$ and $B'$ of $a \in A$ and $f(a) \in B$, respectively,
such that $f:A' \ra B'$ is a diffeomorphism.
\end{rem}

\begin{ex}
Note that any smooth linear bijection from a diffeological vector space
to a fine diffeological vector space is a linear diffeomorphism.
As a consequence, we know that the tangent space $T_0^{dvs}(Y)$ of the axes $Y$
in $\R^2$ with the sub-diffeology (Example~\ref{ex:cross-sub}) has the fine diffeology.
\end{ex}

\begin{prop}\label{prop:finedvs}
Let $V$ be a fine diffeological vector space.
Then $T_0(V) \cong V$ as diffeological vector spaces.
\end{prop}

\begin{proof}
Recall that the vector space $T_0(V)$ is the colimit of $T_0(U)$
taken over the category $\cG(V,0)$.
Consider the full subcategory $\cG$ of $\cG(V,0)$
consisting of inclusions $W \hookrightarrow V$ of finite-dimensional linear subspaces.
It is not hard to see that this is a final subcategory, in the sense that
the overcategory $p/\cG$ is non-empty and connected for each
object $p$ in $\cG(V,0)$.
Therefore, we can restrict to the subcategory $\cG$
and find that, as vector spaces,
\[
  T_0(V) \cong \colim_{W \in \cG} T_0(W) \cong \colim_{W \in \cG} W \cong V.
\]
By the criteria in Proposition~\ref{prop:fine}, we know that $T_0^{dvs}(V)$
is a fine diffeological vector space, and so $T_0^{dvs}(V) \cong V$ as
diffeological vector spaces.
By Theorem~\ref{thm:Tdiffgrptrivial},
$T_0^H(V) = T_0^{dvs}(V)$ as diffeological vector spaces.
\end{proof}

\begin{rem}
When $V$ is fine, it follows from Theorem~\ref{thm:Tdiffgrptrivial} that
$TV \cong V \times V$ as
diffeological vector spaces over $V$.
This is not true for an arbitrary diffeological vector space;
see Example~\ref{(in)dis}(2).
\end{rem}

\begin{ex}\label{ex:notfine}
The following examples show that when $G$ is a diffeological group,
$T_e(G)$ is not necessarily fine, contradicting the argument given
in~\cite[Proposition~6.8]{HM}.
As a consequence, although the vector space $T_e(G)$ is a colimit in the category
of vector spaces of $T_0(U)$, the diffeological vector space $T_e(G)$ is not a colimit in the category
of diffeological vector spaces of $T_0(U)$ with the fine diffeology.
For properties of (fine) diffeological vector spaces, see~\cite{Wu}.

\begin{enumerate}
\item \label{item:inftyprod} Let $X = \prod_\omega \R$ be the countable product of copies of $\R$
with the product diffeology.
Then there is a canonical smooth map $TX \ra \prod_\omega T \R$
which induces a smooth linear map $\alpha:T_0(X) \ra \prod_\omega \R$,
where we identify $T_0 (\R)$ with $\R$ in the natural way.
Consider the plot $p: \R^2 \to X$ sending $(s,t)$ to $(t, st, s^2t, s^3t, \ldots)$.
It induces one of Hector's generating plots $T\R^2 \to TX$.
The restriction of $T \R^2$ to the $s$-axis is a rank-$2$ bundle
whose total space $U$ is diffeomorphic to $\R^3$.
The composite $U \hookrightarrow T\R^2 \to TX$ gives a plot $q$ in $T_0(X)$,
since the plot $p$ sends the $s$-axis to the point $0$ in $X$.
Now consider the composite $\alpha \circ q: U \ra T_0(X) \ra \prod_\omega \R$.
At the point $(s,0)$, it sends $\partial_s$ to $0$
and $\partial_t$ to $(1, s, s^2, s^3, \ldots)$.
As $s$ varies, these $(1, s, s^2, s^3, \ldots)$'s are all
linearly independent in $\prod_\omega \R$.
Since $\alpha$ is linear,
$q$ doesn't locally factor through any finite-dimensional linear subspace of $T_0(X)$.
Therefore, the diffeology on $T_0(X)$ is strictly larger than the fine diffeology.
It follows from Theorem~\ref{thm:Tdiffgrptrivial} that the same is true for $T_x(X)$ for any $x \in X$,
since $X$ is a diffeological group (in fact, a diffeological vector space).

While it isn't needed above, one can also show that $\alpha$ is an
isomorphism, so $T_0(X) \cong X$ as vector spaces.
The surjectivity follows from Remark~\ref{rem:infinite-product},
and the injectivity follows from the fact that $X$ is a diffeological group,
using an argument similar to that used in Example~\ref{halfline:sub}.
In fact, Proposition~\ref{prop:isots} shows that $\alpha$ is a diffeomorphism.

\item Write $Y$ for $C^\infty(\R^n,\R)$, where $n \geq 1$,
and write $\hat{0}$ for the zero function in $Y$.
Define $\phi:Y \ra X$ by $f \mapsto (f(0),\frac{\partial f}{\partial y_1}(0),\frac{\partial^2 f}{\partial y_1^2}(0),\ldots)$.
Then $\phi$ is a smooth map, by~\cite[Lemma~4.3]{CSW}.
Let $p:\R^2 \ra Y$ be defined by $p(s,t)(y)=te^{sy_1}$.
Then $p$ is a plot of $Y$.
The restriction of $T\R^2$ to the $s$-axis is a rank-$2$ bundle
whose total space $U$ is diffeomorphic to $\R^3$.
As in~(1), one can check that the composite
$U \hookrightarrow T\R^2 \ra TY$ gives a plot $q$ in $T_{\hat{0}}(Y)$,
but the composite
$\alpha \circ T_{\hat{0}}\phi \circ q:
U \ra T_{\hat{0}}(Y) \ra T_0(X) \ra \prod_\omega \R$
does not locally factor through a finite-dimensional linear subspace
of $\prod_\omega \R$.
Therefore, $T_f(Y)$ is not a
fine diffeological vector space for any $f \in Y$.

\item Let $M$ be a smooth manifold of positive dimension,
and let $i:\R^n \ra M$ be a smooth chart.
Using a smooth partition of unity one can show that for any smooth map $f:\R^2 \times \R^n \ra \R$,
there exists a smooth map $g:\R^2 \times M \ra \R$ such that
$g \circ (1 \times i)|_{\R^2 \times B}=f|_{\R^2 \times B}$,
where $B$ is an open neighborhood of $0$ in $\R^n$.
By (2), we know that $T_f(C^\infty(M,\R))$
is not a fine diffeological vector space
for any $f \in C^\infty(M,\R)$.
\end{enumerate}

As a corollary of (3),
if $M$ is a compact smooth manifold of positive dimension,
and $U$ is an open subset of $\R^n$ for some $n \in \Z^+$,
then for any $f \in C^\infty(M,U)=:Z$,
both $T^H_f(Z)$ and $T^{dvs}_f(Z)$ are not fine.
This follows from
Proposition~\ref{prop:subandproductforbundles}(1) and the fact that
$Z=C^\infty(M,U)$ is a $D$-open subset of
$C^\infty(M,\R^n)$~\cite[Proposition~4.2]{CSW}.
\end{ex}

\subsection{Internal tangent bundles of function spaces}\label{ss:function}

In this subsection, we first observe in Proposition~\ref{prop:map}
that for any diffeological spaces $X$ and $Y$,
there is always a natural smooth map $\gamma : T^H(C^\infty(X,Y)) \ra C^\infty(X,T^H (Y))$.
In Propositions~\ref{prop:isots} and~\ref{prop:cpt}, we give two special cases when
the above map is actually a diffeomorphism,
and it follows that in these cases $T^H(C^\infty(X,Y))=T^{dvs}(C^\infty(X,Y))$.
In particular, we recover in Corollary~\ref{cor:cpt}
the fact that the internal tangent space
of the diffeomorphism group of a compact smooth manifold at the identity
is isomorphic to the vector space of all smooth vector fields on the manifold.

Stacey~\cite{St2} also studies the map $\gamma$, but with a different focus.

\begin{prop}\label{prop:map}
Let $X$ and $Y$ be diffeological spaces.
There is a smooth map
\[
\gamma:T^H(C^\infty(X,Y)) \lra C^\infty(X,T^H (Y)),
\]
which is natural in $X$ and $Y$,
and which makes the following triangle commutative:
\[
\xymatrix@C5pt{T^H(C^\infty(X,Y)) \ar[rr]^\gamma \ar[dr]_{\pi_{C^\infty(X,Y)}}
& & C^\infty(X,T^H(Y)) \ar[dl]^{(\pi_Y)_*} \\ & C^\infty(X,Y).}
\]
\end{prop}
\begin{proof}
Observe that the map $\tau:C^\infty(\R,Y) \ra T^H(Y)$ defined by
$\alpha \mapsto (\alpha(0),\alpha_* (\frac{d}{dt}))$ is smooth,
since for any plot $q:U \ra C^\infty(\R,Y)$,
we have a commutative diagram
\[
\xymatrix{U \ar[d] \ar[r]^-q & C^\infty(\R,Y) \ar[d]^-\tau \\
          TU \times T \R \ar[r]_-{T \tilde{q}} & T^H(Y)}
\]
in $\Diff$,
where the left vertical map is defined by $u \mapsto ((u,0),(0,\frac{d}{dt}))$,
and $\tilde{q}:U \times \R \ra Y$ is the adjoint of $q$.

Now we first partially define $\gamma:T^H(C^\infty(X,Y)) \ra C^\infty(X,T^H (Y))$
by sending $(f,p_*(\frac{d}{dt}))$ to $\tau \circ \hat{p}$,
where $p:\R \ra C^\infty(X,Y)$ is a plot such that $p(0)=f$,
and $\hat{p}$ is the double adjoint of $p$,
using the cartesian closedness of $\Diff$.
One can show that the following triangle
\[
\xymatrix@C5pt{T^H(C^\infty(X,Y)) \ar[rr]^\gamma \ar[dr]_{\pi_{C^\infty(X,Y)}} & & C^\infty(X,T^H(Y)) \ar[dl]^{(\pi_Y)_*} \\ & C^\infty(X,Y),}
\]
commutes, so we can linearly extend this on each fibre to define $\gamma$.
It is straightforward to check that the map $\gamma$ is well-defined.
Moreover, $\gamma$ is smooth:
given any plots $q:U \ra C^\infty(X,Y)$ and $r:V \ra X$, the composite
\[
\xymatrix{TU \ar[r]^-{Tq} & T^H(C^\infty(X,Y)) \ar[r]^\gamma & C^\infty(X,T^H(Y))}
\]
is smooth since for the adjoint map $TU \times X \ra T^H(Y)$
we have a commutative diagram
\[
\xymatrix{TU \times V \ar[r]^{1 \times r} \ar[d]_{1 \times \sigma} & TU \times X \ar[d] \\ TU \times TV \ar[r]_{T \beta} & T^H(Y),}
\]
where $\sigma$ is the zero section, and $\beta$ is the composite
\[
\xymatrix{U \times V \ar[r]^{1 \times r} & U \times X \ar[r]^-{\tilde{q}} & Y.}
\]
The naturality of $\gamma$ directly follows from its definition.
\end{proof}

\begin{aside}
Note that when $X$ is discrete, the map $\gamma$ just defined
coincides with the map $\alpha$ from Remark~\ref{rem:infinite-product}
in the case where each $X_j$ is equal to $Y$.
In fact, there is a framework which encompasses both the map $\gamma$ in the previous proposition
and the map $\alpha$ in Remark~\ref{rem:infinite-product} as special cases.
Let $f:Y \ra X$ be a fixed smooth map between diffeological spaces.
Write $\Gamma(f)$ for the set of all smooth sections of $f$ equipped with the sub-diffeology of $C^\infty(X,Y)$,
and write $A$ for the set $\{g \in C^\infty(X,T^H(Y)) \mid \pi_Y \circ g \in \Gamma(f) \text{ and } f_*(g(x))=(x,0) \text{ for all } x \in X\}$ of ``vertical'' sections of $T^H(Y) \to X$,
equipped with the sub-diffeology of $C^\infty(X,T^H(Y))$.
Then $(\pi_Y)_*:C^\infty(X,T^H(Y)) \ra C^\infty(X,Y)$ restricts to a smooth map $(\pi_Y)_*:A \ra \Gamma(f)$.
Moreover, $\gamma:T^H(C^\infty(X,Y)) \ra C^\infty(X,T^H(Y))$ in the previous proposition induces a smooth map $\gamma:T^H(\Gamma(f)) \ra A$,
and we have a commutative triangle
\[
\xymatrix@C5pt{T^H(\Gamma(f)) \ar[rr]^-{\gamma} \ar[dr]_{\pi_{\Gamma(f)}} && A \ar[dl]^{(\pi_Y)_*} \\ & \Gamma(f).}
\]
In particular,
if the fixed map $f$ is the projection $\pr_1:X \times Y \ra X$, then we recover the previous proposition.
And if $J$ is a discrete diffeological space, $\{X_j\}_{j \in J}$ is a set of diffeological spaces,
and the fixed map $f: \coprod_{j \in J} X_j \ra J$ sends each $X_j$ to $j$,
then the map $\gamma:T^H(\Gamma(f)) \ra A$ discussed above is the same as the map $\alpha$ discussed in Remark~\ref{rem:infinite-product}.
\end{aside}

We will now focus on cases in which $T^H(Y) = T^{dvs}(Y)$, so we write $TY$
to simplify the notation.
In this situation, it is straightforward to show that
$(\pi_Y)_*: C^\infty(X,TY) \ra C^\infty(X,Y)$
is a diffeological vector space over $C^\infty(X,Y)$.
Given a smooth map $f : X \to Y$, we write
$S_f(X,Y)=\{ g \in C^\infty(X, \, TY) \mid \pi_Y \circ g=f \}$
for the fibre of $(\pi_Y)_*$ over $f$, a diffeological vector space.
The map $\gamma$ restricts to a natural smooth linear map
$\gamma_f: T_f^H(C^\infty(X,Y)) \ra S_f(X,Y)$.

\begin{prop}\label{prop:isots}
Let $X$ be a diffeological space.
The map $\gamma:T(C^\infty(X,\R^n)) \ra C^\infty(X,T\R^n)$
from Proposition~\ref{prop:map} is an isomorphism of diffeological
vector spaces over $C^\infty(X,\R^n)$.
\end{prop}
In particular, taking the case where $X$ is discrete shows that
$T(\prod_{j \in X} \R^n) \cong \prod_{j \in X} T\R^n$ as diffeological vector spaces over $\prod_{j \in X} \R^n$;
see Proposition~\ref{product}, Remark~\ref{rem:infinite-product}
and Example~\ref{ex:notfine}(1).

\begin{proof}
We first prove that for each $f \in C^\infty(X,\R^n)$,
the restriction $\gamma_f:T_f(C^\infty(X,\R^n)) \ra S_f(X,\R^n)$ of $\gamma$
to the fibre over $f$ is an isomorphism of vector spaces.

Since in this case, $\gamma$ is a smooth map between diffeological
vector spaces over $C^\infty(X,\R^n)$,
it is enough to prove that $\gamma_f$ is a bijection for $n=1$.

We first prove injectivity.
Note that $C^\infty(X,\R)$ is a diffeological group.
Corollary~6.5 of~\cite{HM} says that for any diffeological group $G$,
every element of $T_e(G)$ comes from a plot $U \ra G$.
(Compare with Example~\ref{ex:bundleofcross}.)
Since $T_e(G) \cong T_g(G)$, the same is true for $T_g(G)$.
By the proof of Proposition~\ref{lgcurve}, every internal tangent vector
is in fact of the form $p_*(\frac{d}{dt})$, where $p : \R \to G$ is
a plot with $p(0) = g$.
Also note that:

(1) If $\alpha:U \times X \ra \R$ is smooth
for $U$ some open subset of a Euclidean space,
then $\frac{\partial \alpha}{\partial u_i}:U \times X \ra \R$
is also smooth for any coordinate $u_i$ in $U$,
since for any plot $\beta:W \ra X$, $\frac{\partial \alpha}{\partial u_i} \circ (1_U \times \beta)=\frac{\partial}{\partial u_i}(\alpha \circ (1_U \times \beta))$.

(2) If $K$ is a compact subset of $U$,
$\phi:K \ra \R$ is an integrable function,
and $F:U \times X \ra \R$ is smooth,
then by~\cite[Theorem~V.2.9.9]{G},
the function $X \ra \R$ defined by $\int_K \phi(s) F(s,x) \, ds$ is smooth.

Now assume that a plot $p:\R \ra C^\infty(X,\R)$
with $p(0)=f$ induces $0 \in S_f(X,\R)$;
that is, the adjoint map $\tilde{p}:\R \times X \ra \R$ is smooth,
$\tilde{p}(0,x)=f(x)$ and $\frac{\partial \tilde{p}}{\partial t}(0,x)=0$
for all $x \in X$.
We can conclude that $\tilde{p}(t,x)=f(x)+tg(t,x)$
for $g(t,x)=\int_0^1 (D_1 \tilde{p})(st,x) \, ds$,
and this $g$ is in $C^\infty(\R \times X,\R)$ by (1) and (2).
Note that $g(0,x)=0$ for all $x \in X$.
Then we define $q : \R^2 \ra C^\infty(X,\R)$ by
$q(t_1,t_2)(x) = f(x)+t_1 g(t_2,x)$.
The restriction of $q$ to
either axis is a constant map $\R \ra C^\infty(X,\R)$,
so $q_*(\frac{d}{dt_1}) = 0 = q_*(\frac{d}{dt_2})$ in $T_f(C^\infty(X,\R))$.
And the restriction of $q$ to the diagonal is $p$, so,
as in the argument given in Example~\ref{halfline:sub},
$p_*(\frac{d}{dt}) = 0$ in $T_f(C^\infty(X,\R))$.

For surjectivity,
take any $(f,g) \in S_f(X,\R)$; that is $g:X \ra \R$ is any smooth map.
Define $p:\R \ra C^\infty(X,\R)$ by $t \mapsto (x \mapsto f(x)+tg(x))$.
It is straightforward to check that $p$ is a plot with $p(0)=f$ and
that the map $\gamma_f:T_f(C^\infty(X,\R)) \ra S_f(X,\R)$
sends $p_*(\frac{d}{dt})$ to $g$.

Hence, together with Proposition~\ref{prop:map},
we have proved so far that the  map $\gamma:T(C^\infty(X,\R^n)) \ra C^\infty(X,T\R^n)$
is a smooth bijection and linear on each fibre.
We claim that in this case, $\gamma^{-1}$ is also smooth, and hence
$\gamma$ is an isomorphism of diffeological vector spaces over $C^\infty(X,\R^n)$.
Here is the proof.
Notice that it is enough to prove this for $n=1$.
For any plot $q=(q_1,q_2):U \ra C^\infty(X,T\R)$, we define
a smooth map $\rho:\R \times U \ra C^\infty(X,\R)$
by $\rho(t,u)=q_1(u)+tq_2(u)$.
The smoothness of $\gamma^{-1}$ follows from the commutative diagram
\[
\xymatrix{U \ar[d] \ar[r]^-q & C^\infty(X,T\R) \ar[d]^(0.45){\gamma^{-1}} \\ T\R \times TU \ar[r]_-{T\rho} & T(C^\infty(X,\R)),}
\]
where the left vertical map is given by $u \mapsto ((0,\frac{d}{dt}),(u,0))$.
\end{proof}

As direct consequences,
(i) the map $\gamma_f:T_f(C^\infty(X,\R^n)) \ra S_f(X,\R^n)$ is
an isomorphism between diffeological vector spaces;
(ii) $C^\infty(X,\R^n) \times C^\infty(X,\R^n) \cong C^\infty(X,T\R^n) \cong
T(C^\infty(X,\R^n))$ as diffeological vector spaces over $C^\infty(X,\R^n)$.
Moreover, we have:

\begin{prop}\label{prop:cpt}
Let $X$ be a diffeological space such that its $D$-topology is compact,
and let $N$ be a smooth manifold.
Then the map $\gamma:T^H(C^\infty(X,N)) \ra C^\infty(X,TN)$ in
Proposition~\ref{prop:map} is an isomorphism of diffeological vector
spaces over $C^\infty(X,N)$.
\end{prop}

\begin{proof}
Let $N \ra \R^n$ be a smooth embedding,
and let $U$ be an open tubular neighborhood of $N$
with inclusions $i:N \ra U$ and $j:U \ra \R^n$,
and a smooth retraction $r:U \ra N$ such that $r \circ i=1_N$.
Since the $D$-topology on $X$ is compact,~\cite[Proposition~4.2]{CSW} implies that
$C^\infty(X,U)$ is a $D$-open subset of $C^\infty(X,\R^n)$,
and hence by Proposition~\ref{prop:subandproductforbundles}(1),
$j_*:T^H(C^\infty(X,U)) \ra T(C^\infty(X,\R^n))$ is an induction, and
for each $f \in C^\infty(X,U)$, the restriction
$j_{*,f}:T_f^H(C^\infty(X,U)) \ra T_f(C^\infty(X,\R^n))$ is an isomorphism.
It is straightforward to see that $j_*:C^\infty(X,TU) \ra C^\infty(X,T\R^n)$
is also an induction, and for each $f \in C^\infty(X,U)$,
the restriction $j_{*,f}:S_f(X,U) \ra S_f(X,\R^n)$
is an isomorphism, since $U$ is open in $\R^n$.
So by Proposition~\ref{prop:isots} and the functoriality of $\gamma$,
we know that $\gamma^U:T^H(C^\infty(X,U)) \ra C^\infty(X,TU)$ is an isomorphism
of diffeological vector spaces over $C^\infty(X,U)$.
Since $\gamma:T^H(C^\infty(X,N)) \ra C^\infty(X,TN)$ is a retract of $\gamma^U$,
it is also an isomorphism of diffeological vector spaces over $C^\infty(X,N)$.
\end{proof}

In this case, since $T^H(C^\infty(X,N))$ is
a diffeological vector space over $C^\infty(X,N)$,
it follows that $T^H(C^\infty(X,N))=T^{dvs}(C^\infty(X,N))$.

Note that, if $M$ is a smooth manifold,
then $\Diff(M)$, the set of all diffeomorphisms from $M$ to itself,
equipped with the sub-diffeology of $C^\infty(M,M)$,
is a diffeological group.
So we recover Proposition~6.3 of~\cite{HM}:

\begin{cor}\label{cor:cpt}
Let $M$ be a compact smooth manifold.
Then $T_{1_M}(\Diff(M))$ is isomorphic to the vector space of
all smooth vector fields on $M$.
\end{cor}
\begin{proof}
This follows from Propositions~\ref{prop:internallocal} and~\ref{prop:cpt},
and the fact that $\Diff(M)$ is a $D$-open subset of $C^\infty(M,M)$
since $M$ is a compact smooth manifold~\cite[Corollary~4.15]{CSW}.
\end{proof}

\vspace*{10pt} %

\begin{thebibliography}{CSW}

    \bibitem[CSW]{CSW} {\scshape J.D. Christensen, J.G. Sinnamon, and E. Wu},
    \emph{The $D$-topology for diffeological spaces},
    Pacific J. Math. \textbf{272}(1) (2014), 87--110.

    \bibitem[CW1]{CW} {\scshape J.D. Christensen and E. Wu},
    \emph{The homotopy theory of diffeological spaces},
    New York J. Math. \textbf{20} (2014), 1269--1303.

    \bibitem[CW2]{CW2} {\scshape J.D. Christensen and E. Wu},
    \emph{Tangent spaces of bundles and of filtered diffeological spaces},
    preprint, available at \url{http://arxiv.org/abs/1510.09182}.

    \bibitem[G]{G} {\scshape R. Godement},
    \emph{Analysis II},
    Universitext, Springer, 2005.

    \bibitem[He]{He} {\scshape G. Hector},
    \emph{G\'eom\'etrie et topologie des espaces diff\'eologiques},
    in Analysis and Geometry in Foliated Manifolds (Santiago de Compostela, 1994),
    World Sci. Publishing (1995), pp. 55--80.

    \bibitem[HM]{HM} {\scshape G. Hector and E. Macias-Virgos},
    \emph{Diffeological groups},
    Research and Exposition in Mathematics \textbf{25} (2002), pp. 247--260.

    \bibitem[I1]{I1} {\scshape P. Iglesias-Zemmour},
    \emph{Fibrations diff\'eologiques et Homotopie},
    Th\`ese de Doctorat Es-sciences, L'Universit\'e de Provence, 1985.

    \bibitem[I2]{I2} {\scshape P. Iglesias-Zemmour},
    \emph{Dimension in diffeology},
    Indag. Math. \textbf{18}(4) (2007), pp. 555--560.

    \bibitem[I3]{I3} {\scshape P. Iglesias-Zemmour},
    \emph{Diffeology},
    Mathematical Surveys and Monographs, 185, AMS, Providence, 2013.

    \bibitem[KM]{KM} {\scshape A. Kriegl and P.W. Michor},
    \emph{The convenient setting of global analysis},
    Mathematical Surveys and Monographs, 53, American Mathematical Society, 1997.

    \bibitem[La]{La} {\scshape M. Laubinger},
    \emph{Diffeological spaces},
    Proyecciones \textbf{25}(2) (2006), pp. 151--178.

    %
    %
    %

    \bibitem[Mac]{Mac} {\scshape S. Mac\,Lane},
    \emph{Categories for the working mathematician},
    second edition, Graduate Texts in Mathematics, 5, Springer-Verlag, 2003.

    %
    %
    %
    %
    %

    \bibitem[So1]{So1} {\scshape J.-M. Souriau},
    \emph{Groupes diff\'erentiels},
    Differential geometrical methods in mathematical physics
    (Proc. Conf., Aix-en-Provence/Salamanca, 1979),
    Lecture Notes in Math., 836, Springer (1980), pp. 91--128.

    \bibitem[So2]{So2} {\scshape J.-M. Souriau},
    \emph{Groupes diff\'erentiels de physique math\'ematique},
    South Rhone seminar on geometry, II (Lyon, 1983),
    Travaux en Cours, Hermann, Paris, (1984), pp. 73--119.

    %
    %
    %
    %
    %
    %
    %
    %

    \bibitem[St1]{St} {\scshape A. Stacey},
    \emph{Comparative smootheology},
    Theory and Appl. of Categ. \textbf{25}(4) (2011), pp. 64--117.

    \bibitem[St2]{St2} {\scshape A. Stacey},
    \emph{Yet more smooth mapping spaces and their smoothly local properties},
    preprint, available at \url{http://arxiv.org/abs/1301.5493v1}.

    %
    %
    %
    %

    \bibitem[V]{V} {\scshape M. Vincent},
    \emph{Diffeological differential geometry},
    Master's Thesis, University of Copenhagen, 2008,
    available at
    \url{http://www.math.ku.dk/english/research/tfa/top/paststudents/martinvincent.msthesis.pdf}.

    %
    %
    %

    \bibitem[Wa]{Wa} {\scshape J. Watts},
    \emph{Diffeologies, differential spaces, and symplectic geometry},
    preprint, available at \url{http://arxiv.org/abs/1208.3634v2}.

%
%
%

    \bibitem[Wh]{Wh} {\scshape H. Whitney},
    \emph{Differentiable even functions},
    Duke Math. J. \textbf{10} (1943), pp. 159--160.

    \bibitem[Wu]{Wu} {\scshape E. Wu},
    \emph{Homological algebra for diffeological vector spaces},
    Homology Homotopy Appl. \textbf{17}(1) (2015), pp. 339--376.

\end{thebibliography}
\end{document}